\documentclass[dvips,aap,noinfoline]{imsart}

\RequirePackage[OT1]{fontenc}
\RequirePackage{amsthm,amsmath}
\RequirePackage[numbers]{natbib}
\RequirePackage[colorlinks,citecolor=blue,urlcolor=blue]{hyperref}
\RequirePackage{hypernat}

\usepackage{amsfonts, amssymb}
\usepackage{graphics}
\usepackage{epsfig}
\usepackage{psfrag}
\usepackage{subfigure}
\usepackage[rightcaption]{sidecap}

\startlocaldefs
\newcommand{\R}{\ensuremath{\mathbb{R}}}

\newcommand{\C}{\ensuremath{\mathcal{C}}}

\newcommand{\E}{\ensuremath{\mathbb{E}}}

\newcommand{\W}[2]{\ensuremath{W^{(#1)}(#2)}}
\newcommand{\dW}[2]{\ensuremath{W^{(#1)\prime}(#2)}}
\newcommand{\Z}[2]{\ensuremath{Z^{(#1)}(#2)}}
\newcommand{\dZ}[2]{\ensuremath{Z^{(#1)\prime}(#2)}}
\newcommand{\ddZ}[2]{\ensuremath{Z^{(#1)\prime\prime}(#2)}}
\newcommand{\Y}[5]{\ensuremath{\frac{Z^{(#1)}(#2)W^{(#1)\prime}(#3)}{W^{(#1)}(#4)}-#1W^{(#1)}(#5)}}

\theoremstyle{plain}
\newtheorem{thm}{Theorem}[section]
\newtheorem{lem}[thm]{Lemma}

\newtheorem{rem}[thm]{Remark}
\endlocaldefs

\begin{document}


\renewcommand{\P}{\ensuremath{\mathbb{P}}}

\psfrag{x}{\tiny $x$}
\psfrag{s}{\tiny $s$}
\psfrag{O1}{\tiny $0$}
\psfrag{O2}{\tiny $\log(K)$}
\psfrag{k}{\tiny $k^*$}
\psfrag{k1}{\tiny $k^*$}
\psfrag{k2}{\tiny $\log(K)$}
\psfrag{D}{\tiny $D^*$}
\psfrag{H}{\tiny $H$}
\psfrag{Iso}{\tiny $H\mapsto f(H)$}
\psfrag{CI}{\tiny $C^*_I$}
\psfrag{CII}{\tiny $C^*_{II}$}
\psfrag{Sol}{\tiny $g(s)$}
\psfrag{eta}{\tiny $\eta$}
\psfrag{k4}{\tiny $\beta$}
\psfrag{beta}{\tiny $\beta$}
\psfrag{e}{\tiny $\epsilon$}
\psfrag{g_e}{\tiny $g_\epsilon(s)$}
\psfrag{g_i}{\tiny $g_\infty(s)$}

\begin{frontmatter}
\title{A Capped  Optimal Stopping Problem for the Maximum Process}
\runtitle{Optimal stopping for the maximum process}
\author{\fnms{Andreas} \snm{Kyprianou}\ead[label=e1]{a.kyprianou@bath.ac.uk}}
\and
\author{\fnms{Curdin} \snm{Ott}\ead[label=e2]{C.Ott@bath.ac.uk}}
\affiliation{University of Bath}

\begin{abstract}
This paper concerns an optimal stopping problem driven by the running maximum of a spectrally negative L\'evy process $X$. More precisely, we are interested in capped versions of the American lookback optimal stopping problem~\cite{guo_shepp,pedersen, gapeev}, which has its origins in mathematical finance, and provide semi-explicit solutions in terms of scale functions. The optimal stopping boundary is characterised by an ordinary first-order differential equation involving scale functions and, in particular, changes according to the path variation of $X$. Furthermore, we will link these capped problems to Peskir's maximality principle~\cite{maximality_principle}.
\end{abstract}

\address{Mathematical Sciences\\ University of Bath\\ Claverton Down\\ Bath BA2 7AY\\ United Kingdom\\\printead{e1}}
\address{Mathematical Sciences\\ University of Bath\\ Claverton Down\\ Bath BA2 7AY\\ United Kingdom\\\printead{e2}}



\end{frontmatter}

\section{Introduction}
Let $X=\{X_t:t\geq 0\}$ be a spectrally negative L\'evy process defined on a filtered probability space $(\Omega,\mathcal{F},\mathbb{F}=\{\mathcal{F}_t:t\geq 0\},\P)$ satisfying the natural conditions (cf.~p.39, Section 1.3 of~\cite{bichteler}). For $x\in\R$, denote by $\P_x$ the probability measure under which $X$ starts at $x$ and for simplicity write $\P_0=\P$. We associate with $X$ the maximum process \mbox{$\overline X=\{\overline X_t:t\geq 0\}$} where $\overline X_t:=s\vee\sup_{0\leq u\leq t} X_u$ for $t\geq 0,s\geq x$. The law under which $(X,\overline X)$ starts at $(x,s)$ is denoted by $\P_{x,s}$.\\
\indent We are interested in the following optimal stopping problem:
\begin{equation}
V_\epsilon^*(x,s):=\sup_{\tau\in\mathcal{M}}\E_{x,s}\big[e^{-q\tau}(e^{\overline X_\tau\wedge\epsilon}-K)^+\big],\label{problem1}
\end{equation}
where $q\geq0,\epsilon\in(\log(K),\infty],K\geq 0$, $(x,s)\in E$, where
\begin{equation*}
E:=\{(x,s)\in\R^2\,\vert\, x\leq s\},
\end{equation*}
and $\mathcal{M}$ is the set of all $\mathbb{F}$-stopping times (not necessarily finite). In particular, on $\{\tau=\infty\}$ we set $e^{-q\tau}(e^{\overline X_\tau\wedge\epsilon}-K)^+:=\limsup_{t\to\infty}e^{-qt}(e^{\overline X_t\wedge\epsilon}-K)^+$. This problem is, at least in the case $\epsilon=\infty$, classically associated with mathematical finance. It arises in the context of pricing American lookback options~\cite{guo_shepp,pedersen, gapeev} and its solution may be viewed as the fair price for such an option. If $\epsilon\in (\log(K),\infty)$, an analogous interpretation applies for an American lookback option whose payoff is moderated by capping it at a certain level (a fuller description will be given in Section~\ref{application}).\\  
\indent When $K=0$ and $\epsilon=\infty$,~\eqref{problem1} is known as the Shepp-Shiryaev optimal stopping  problem which was first studied by Shepp and Shiryaev~\cite{russian_option, a_new_look} for the case when $X$ is a linear Brownian motion and later by Avram, Kyprianou and Pistorius~\cite{exitproblems} for the case when $X$ is a spectrally negative L\'evy process. If $K=0$ and $\epsilon\in\R$ then the problem is a capped version of the Shepp-Shiryaev optimal stopping problem and was considered by Ott~\cite{maximum_process}. Therefore, our main focus in this paper will be the case $K>0$ which we henceforth assume.\\
\indent Our objective is to solve~\eqref{problem1} for $\epsilon=(\log(K),\infty)$ by a ``guess and verify'' technique and use this to obtain the solution to~\eqref{problem1} when $\epsilon=\infty$ via a limiting procedure. Our work extends and complements results by Guo and Shepp~\cite{guo_shepp}, Pedersen~\cite{pedersen} and Gapeev~\cite{gapeev} all of which solve~\eqref{problem1} for $\epsilon=\infty$ and $X$ a linear Brownian motion or a jump-diffusion.\\
\indent As we shall see, the general theory of optimal stopping~\cite{peskir,optimal_stopping_rules} and the principle of smooth and continuous fit~\cite{mikhalevich,pes_shir,peskir,some_remarks} (and the results in~\cite{guo_shepp,pedersen, gapeev,maximum_process}) strongly suggest that under some assumptions on $q$ and $\psi(1)$, where $\psi$ is the Laplace exponent of $X$, the optimal strategy for~\eqref{problem1} is of the form
\begin{equation}
\tau^*_\epsilon=\inf\{t\geq 0:\overline X_t-X_t\geq g_\epsilon(\overline X_t)\text{ and }\overline X_t>\log(K)\}\label{intro_stop_time}
\end{equation}
for some strictly positive solution $g_\epsilon$ of the differential equation
\begin{equation}
g_\epsilon^\prime(s)=1-\frac{e^s\Z{q}{g_\epsilon(s)}}{(e^s-K)q\W{q}{g_\epsilon(s)}}\quad\text{on $(\log(K),\epsilon)$},\label{ode_intro}
\end{equation}
where $W^{(q)}$ and $Z^{(q)}$ are the so-called $q$-scale functions associated with $X$ (see Section~\ref{preliminaries}). In particular, we will find that the optimal stopping boundary $s\mapsto s-g_\epsilon(s)$ changes shape according to the path variation of $X$. This has already been observed in~\cite{maximum_process} in the case of the capped version of the Shepp-Shiryaev optimal stopping problem. It will also turn out that our solutions exhibit a pattern suggested by Peskir's maximality principle~\cite{maximality_principle}. In fact, we will be able to give a reformulation of our main results in terms of Peskir's maximality principle.\\
\indent We conclude this section with an overview of the paper. In Section~\ref{application} we give an application of our results in the context or pricing capped American lookback options. Section~\ref{preliminaries} is an auxiliary section introducing some necessary notation, followed by Section~\ref{regimes} which gives an overview of the different parameter regimes considered. Sections~\ref{general_observation} and~\ref{candidate_solution} deal with the ``guess'' part of our ``guess and verify'' technique and our main results, which correspond to the ``verify'' part, are presented in Section~\ref{main_results}. The proofs of our main results can then be found in Section~\ref{proofs}. Finally, Section~\ref{example} provides an explicit example under the assumption that $X$ is a linear Brownian motion.

\section{Application to pricing ``capped'' American lookback options}\label{application}
The aim of this section is to give some motivation for studying~\eqref{problem1}.\\
\indent Consider a financial market consisting of a riskless bond and a risky asset. The value of the bond $B=\{B_t:t\geq 0\}$ evolves deterministically such that
\begin{equation}
B_t=B_0e^{rt},\quad B_0>0,r\geq 0,t\geq 0.\label{discount}
\end{equation}
The price of the risky asset is modeled as the exponential spectrally negative L\'eve process
\begin{equation}
S_t=S_0e^{X_t},\quad S_0>0,t\geq 0.\label{model}
\end{equation}
In order to guarantee that our model is free of arbitrage we will assume that $\psi(1)=r$. If $X_t=\mu t+\sigma W_t$, where $W=\{W_t:t\geq 0\}$ is a standard Brownian motion, we get the standard Black-Scholes model for the price of the asset. Extensive empirical research has shown that this (Gaussian) model is not capable of capturing certain features (such as skewness, asymmetry and heavy tails) which are commonly encountered in financial data, for example, returns on stocks. To accommodate for the these problems, an idea, going back to~\cite{merton}, is to replace the Brownian motion as model for the log-price by a general L\'evy process $X$ (cf.~\cite{chan}). Here we will restrict ourselves to the model where $X$ is given by a spectrally negative L\'evy process. This restriction is mainly motivated by analytical tractability. It is worth mentioning, however, that Carr and Wu~\cite{carr} as well as Madan and Schoutens~\cite{madan} have offered empirical evidence to support the case of a model in which the risky asset is driven by a spectrally negative L\'evy process for appropriate market scenarios.\\
\indent A capped American lookback option is an option which gives the holder the right to exercise at any stopping time $\tau$ yielding payouts
\begin{equation*}
L_\tau:=e^{-\alpha\tau}\bigg[\bigg(M_0\vee\sup_{0\leq u\leq\tau}S_u\wedge C\bigg)-K\bigg]^+,\quad C> M_0\geq S_0,\alpha\geq0.
\end{equation*}
The constant $M_0$ can be viewed as representing the ``starting'' maximum of the stock price (say, over some previous period $(-t_0,0])$. The constant $C$ can be interpreted as cap and moderates the payoff of the option. The value $C=\infty$ is also allowed and correspond to no moderation at all. In this case we just get a normal American lookback option. Finally, when $C=\infty$ it is necessary to choose $\alpha$ strictly positive to guarantee that it is optimal to stop in finite time and that the value is finite (cf. Theorem~\ref{main_result_1}).\\
\indent Standard theory of pricing American-type options~\cite{shiryaev} directs one to solving the optimal stopping problem
\begin{equation}
V_r(M_0,S_0,C):=B_0\sup_{\tau}\E\big[B^{-1}_\tau L_\tau]\label{motivation}
\end{equation}
where the supremum is taken over all $\mathbb{F}$-stopping times. In other words, we want to find a stopping time which optimizes the expected discounted claim. The right-hand side of~\eqref{motivation} may be rewritten as
\begin{equation*}
\sup_{\tau}\E_{x,s}\big[e^{-q\tau}(e^{\overline X_\tau\wedge\epsilon}-K)^+],
\end{equation*}
where $q=r+\alpha,x=\log(S_0),s=\log(M_0)$ and $\epsilon=\log ( C )$.

\section{Preliminaries}\label{preliminaries}
It is well known that a spectrally negative L\'evy process $X$ is characterised by its L\'evy triplet $(\gamma,\sigma,\Pi)$, where $\sigma\geq0, \gamma\in\R$ and $\Pi$ is a measure on $(-\infty,0)$ satisfying the condition $\int_{(-\infty,0)}(1\wedge x^2)\,\Pi(dx)<\infty$. By the L\'evy-It\^o decomposition, the latter may be represented in the form
\begin{equation}
X_t=\sigma B_t-\gamma t+X^{(1)}_t+X^{(2)}_t,\label{LevyItodecomposition1}
\end{equation}
where $\{B_t:t\geq 0\}$ is a standard Brownian motion, $\{X^{(1)}_t:t\geq 0\}$ is a compound Poisson process with discontinuities of magnitude bigger than or equal to one and $\{X_t^{(2)}:t\geq 0\}$ is a square integrable martingale with discontinuities of magnitude strictly smaller than one and the three processes are mutually independent. In particular, if $X$ is of bounded variation, the decomposition reduces to
\begin{equation}
X_t=\mathtt{d}t-\eta_t\label{LevyItodecomposition2}
\end{equation}
where $\mathtt{d}>0$ and $\{\eta_t:t\geq 0\}$ is a driftless subordinator. Further let
\begin{equation*}
\psi(\theta):=\E\big[e^{\theta X_1}\big],\qquad\theta\geq 0,
\end{equation*}
be the Laplace exponent of $X$ which is known to take the form
\begin{equation*}
\psi(\theta)=-\gamma\theta+\frac{1}{2}\sigma^2\theta^2+\int_{(-\infty,0)}\big(e^{\theta x}-1-\theta x1_{\{x>-1\}}\big)\,\Pi(dx).
\end{equation*}
Moreover, $\psi$ is strictly convex and infinitely differentiable and its derivative at zero characterises the asymptotic behavior of $X$. Specifically, $X$ drifts to $\pm\infty$ or oscillates according to whether $\pm\psi^\prime(0+)>0$ or, respectively, $\psi^\prime(0+)=0$.
The right-inverse of $\psi$ is defined by
\begin{equation*}
\Phi(q):=\sup\{\lambda\geq 0:\psi(\lambda)=q\}
\end{equation*}
for $q\geq 0$.

For any spectrally negative L\'evy process having \mbox{$X_0=0$} we introduce the family of martingales
\begin{equation}
\exp(cX_t-\psi(c)t),
\label{*}
\end{equation}
defined for any $c\in\R$ for which $\psi(c)=\log\E[\exp(cX_1)]<\infty$, and further the corresponding family of measures $\{\P^c\}$ with Radon-Nikodym derivatives
\begin{equation}
\frac{d\P^c}{d\P}\bigg\vert_{\mathcal{F}_t}=\exp(cX_t-\psi(c)t).\label{changeofmeasure}
\end{equation}
For all such $c$ the measure $\P^c_x$ will denote the translation of $\P^c$ under which $X_0=x$. In particular, under $\P_x^c$ the process $X$ is still a spectrally negative L\'evy process (cf. Theorem 3.9 in~\cite{kyprianou}).

A special family of functions associated with spectrally negative L\'evy processes is that of scale functions (cf.~\citep{kyprianou}) which are defined as follows. For $q\geq 0$, the $q$-scale function $W^{(q)}:\R\longrightarrow[0,\infty)$ is the unique function whose restriction to $(0,\infty)$ is continuous and has Laplace transform
\begin{equation*}
\int_0^\infty e^{-\theta x}\W{q}{x}\,dx=\frac{1}{\psi(\theta)-q},\quad\theta>\Phi(q),
\end{equation*}
and is defined to be identically zero for $x\leq 0$. Equally important is the scale function $Z^{(q)}:\R\longrightarrow[1,\infty)$ defined by
\begin{equation*}
\Z{q}{x}=1+q\int_0^x\W{q}{z}\,dz.
\end{equation*}
The passage times of $X$ below and above $k\in\R$ are denoted by
\begin{equation*}
\tau_k^-=\inf\{t>0:X_t\leq k\}\quad\text{and}\quad\tau_k^+=\inf\{t>0:X_t\geq k\}.
\end{equation*}
We will make use of the following two identities (cf.~\citep{exitproblems}). For $q\geq 0$ and $x\in(a,b)$ it holds that
\begin{eqnarray}
&&\E_x\big[e^{-q\tau^+_b}I_{\{\tau^+_b<\tau^-_a\}}\big]=\frac{\W{q}{x-a}}{\W{q}{b-a}},\label{scale1}\\
&&\E_x\big[e^{-q\tau^-_a}I_{\{\tau^+_b>\tau^-_a\}}\big]=\Z{q}{x-a}-\W{q}{x-a}\frac{\Z{q}{b-a}}{\W{q}{b-a}}.\label{scale2}
\end{eqnarray}
For each $c\geq 0$ we denote by $W_c^{(q)}$ the $q$-scale function with respect to the measure $\P^c$. A useful formula (cf.~\citep{kyprianou}) linking the scale function under different measures is given by
\begin{equation}
\W{q}{x}=e^{\Phi(q)x}W_{\Phi(q)}(x)\label{scale5}
\end{equation} 
for $q\geq 0$ and $x\geq 0$.

We conclude this section by stating some known regularity properties of scale functions (cf.~\citep{KuzKypRiv}).

\noindent\textsl{Smoothness}: For all $q\geq 0$,
\begin{equation*}\label{smoothness}
W^{(q)}\vert_{(0,\infty)}\in\begin{cases}
C^1(0,\infty),&\text{if $X$ is of bounded variation and $\Pi$ has no atoms},\\C^1(0,\infty),&\text{if $X$ is of unbounded variation and $\sigma=0$},\\C^2(0,\infty),&\text{$\sigma>0$}.
\end{cases}
\end{equation*}
\\
\textsl{Continuity at the origin:} For all $q\geq 0$,
\begin{equation}\label{continuityatorigin}
\W{q}{0+}=\begin{cases}\mathtt{d}^{-1},&\text{if $X$ is of bounded variation,}\\0,&\text{if $X$ is of unbounded variation.}
\end{cases}
\end{equation}
\\
\textsl{Right derivative at the origin:} For all $q\geq 0$,
\begin{equation}\label{derivativeatorigin}
W^{(q)\prime}_+(0+)=\begin{cases}
\frac{q+\Pi(-\infty,0)}{\mathtt{d}^2},&\text{if $\sigma=0$ and $\Pi(-\infty,0)<\infty$,}\\
\frac{2}{\sigma^2},&\text{if $\sigma>0$ or $\Pi(-\infty,0)=\infty$,}\
\end{cases}
\end{equation}
where we understand the second case to be $+\infty$ when $\sigma=0$.\\
\indent For technical reasons, we require for the rest of the paper that $W^{(q)}$ is in $C^1(0,\infty)$ (and hence $Z^{(q)}\in C^2(0,\infty)$). This is ensured by henceforth assuming  that $\Pi$ is atomless whenever $X$ is of bounded variation.

\section{The different parameter regimes}\label{regimes}
Our analysis distinguishes between the following parameter regimes.\\
\indent\textit{Main cases:}
\begin{itemize}
\item $q>0$ and $\epsilon\in(\log(K),\infty)$.
\item $q>0\vee\psi(1)$ and $\epsilon=\infty$,
\end{itemize}

\textit{Special cases:}
\begin{itemize}
\item $q=0$ and $\epsilon\in(\log(K),\infty)$,
\item $q=0$ and $\epsilon=\infty$,
\item $0<q\leq\psi(1)$ and $\epsilon=\infty$.
\end{itemize}

\section{Candidate solution for the main cases}\label{general_observation} 
The aim of this section is to derive a candidate solution to~\eqref{problem1} for the main cases via the principle of smooth and continuous fit~\cite{mikhalevich,pes_shir,peskir,some_remarks}.\\
\indent We begin by heuristically motivating a class of stopping times in which we will look for the optimal stopping time under the assumption that $q>0$ and $\epsilon\in(\log(K),\infty)$. Because $e^{-qt}(e^{\overline X_t\wedge\epsilon}-K)^+=0$ as long as $(X,\overline X)$ is in the set 
\begin{equation*}
C^*_{II}:=\{(x,s)\in E:s\leq\log(K)\},
\end{equation*}
it is intuitively clear that it is never optimal to stop the process $(X,\overline X)$ \mbox{in $C^*_{II}$}. Moreover, as the process $(X,\overline X)$ can only move upwards by climbing up the diagonal in the $(x,s)$-plane (see Fig.~\ref{intuition}), it can only leave $C^*_{II}$ through the point $(\log(K),\log(K))$. Therefore, one should not exercise until the process $(X,\overline X)$ has exceeded the point $(\log(K),\log(K))$. It is possible that this never happens as $X$ might escape to $-\infty$ before reaching level $\log(K)$. 
\begin{figure}[h]
\includegraphics[scale=0.45]{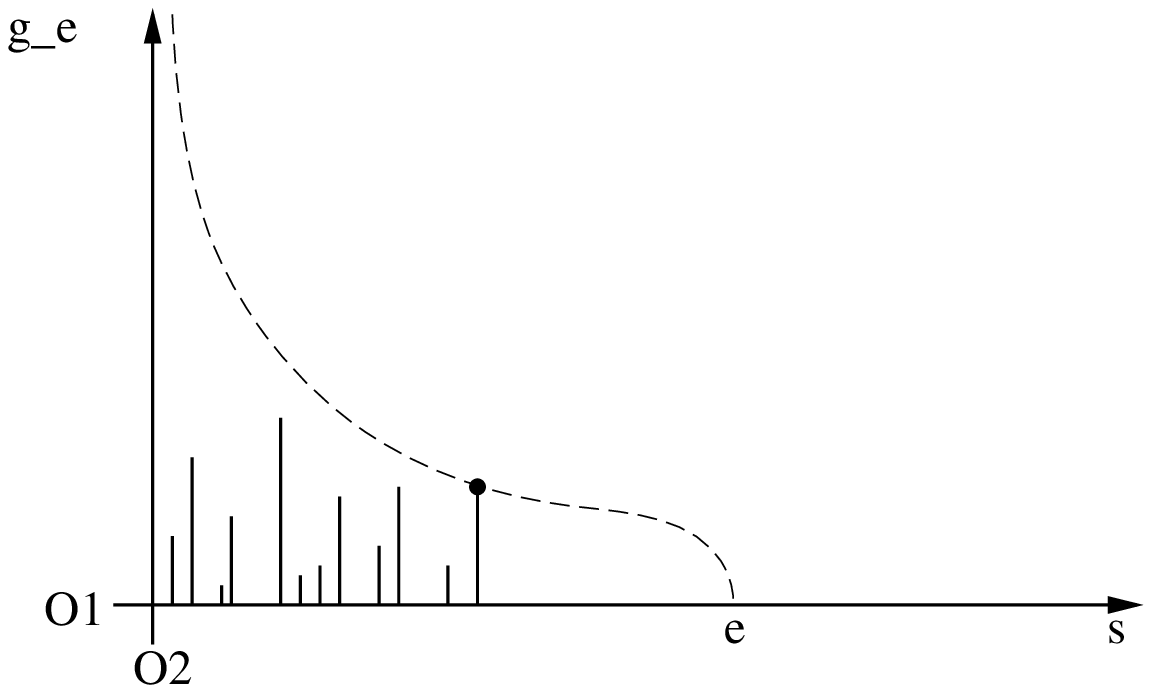}\quad\quad
\includegraphics[scale=0.45]{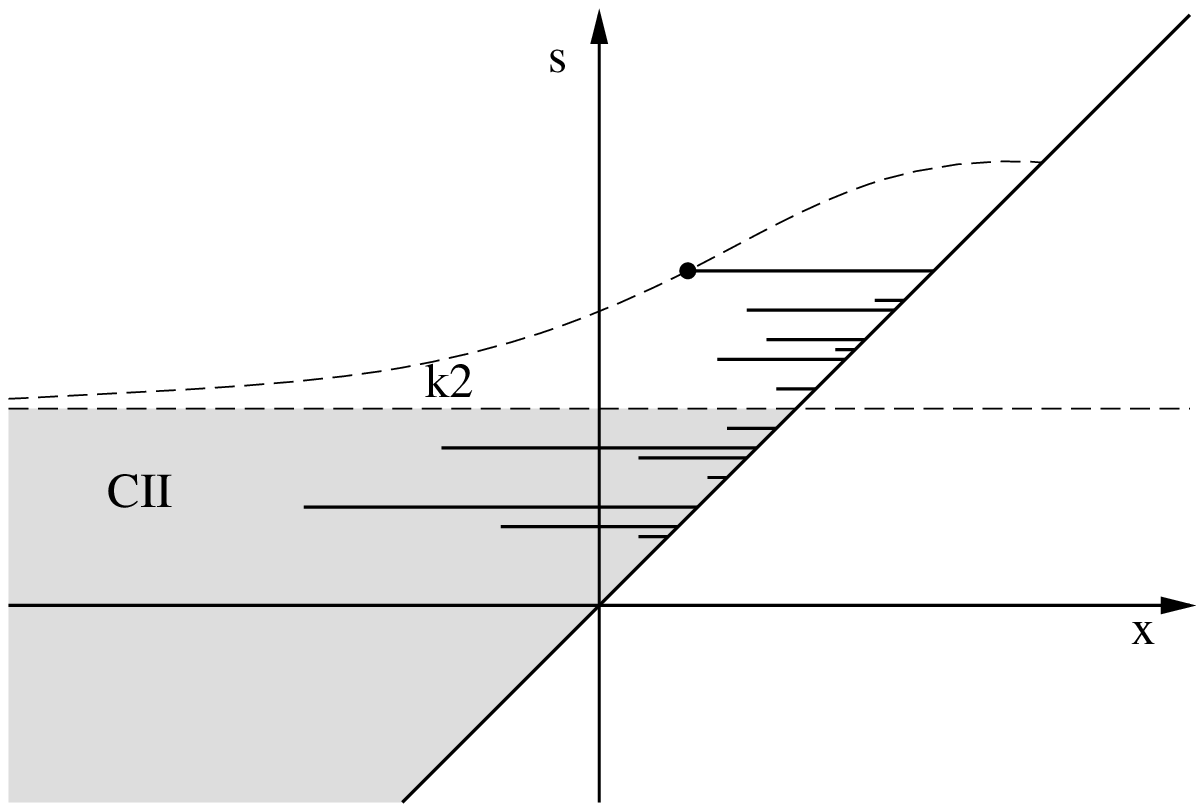}
\caption{An illustration of a possible function $g_\epsilon$ and the corresponding stopping boundary $s\mapsto s-g_\epsilon(s)$. The vertical and horizontal lines are meant to schematically indicate the trace of an excursion of $X$ away from the running maximum. The candidate optimal strategy $\tau_{g_\epsilon}$ then consists of continuing if the height of the excursion away from the running maximum $s$ does not exceed $g_\epsilon(s)$, otherwise we stop.}\label{intuition}
\end{figure}
On the other hand, if the process $(X,\overline X)$ is in $\{(x,s)\in E:s\geq\epsilon\}$, it should be stopped immediately due to the discounting as the spatial part of the payout is deterministic and fixed at $e^{\epsilon}-K$ in value. The remaining case is when $(X,\overline X)$ is in $\{(x,s)\in E:\log(K)<s<\epsilon\}$ in which case we can argue in the same way as described on p.~6, Section 3 of~\cite{maximality_principle}: 
The dynamics of the process $(X, \overline{X})$ are such that $\overline{X}$ remains constant at times when $X$ is undertaking an excursion below $\overline{X}$. During such periods the discounting in the payoff is detrimental. One should therefore not allow $X$ to drop too far below $\overline{X}$ in value as otherwise the time it will $X$ take to recover to the value of its previous maximum will prove to be costly in terms of the gain on account of exponential discounting. More specifically, given a current \mbox{value $s$}, $s\in(\log(K),\epsilon)$, of $\overline{X}$, there should be a point $g_\epsilon(s)>0$ such that if the process $X$ reaches or jumps below the value $(s-g_\epsilon(s),s)$ we should stop instantly (see Fig.~\ref{intuition}). In more mathematical terms, we expect an optimal stopping time of the form
\begin{equation}
\tau_{g_\epsilon}:=\inf\{t\geq 0:\overline X_t-X_t\geq g_\epsilon(\overline X_t)\text{ and }\overline X_t>\log(K)\}\label{expect}
\end{equation}
for some function $g_\epsilon:(\log(K),\epsilon)\rightarrow(0,\infty)$ such that $\lim_{s\uparrow\epsilon}g_\epsilon(s)=0$. This is illustrated in Fig.\ref{intuition}. For $(x,s)\in E$, we define the value function associated with $\tau_{g_\epsilon}$ by
\begin{equation}
V_{g_\epsilon}(x,s):=\E_{x,s}\big[e^{-q\tau_{g_\epsilon}}(e^{\overline X_{\tau_{g_\epsilon}}\wedge\epsilon}-K)^+\big].\label{val}
\end{equation}
Now suppose for the moment that we have chosen a function $g_\epsilon$. The strong Markov property and Theorem 3.12 of~\cite{kyprianou} then imply that, for $(x,s)\in C^*_{II}$,
\begin{eqnarray*}
V_{g_\epsilon}(x,s)&=&e^{-\Phi(q)(\log(K)-x)}\E_{\log(K),\log(K)}\big[e^{-q\tau_{g_\epsilon}}(e^{\overline X_{\tau_{g_\epsilon}}\wedge\epsilon}-K)\big]\\
&=&e^{-\Phi(q)(\log(K)-x)}\lim_{s\downarrow\log(K)}V_{g_\epsilon}(s,s).
\end{eqnarray*}
 This means that $V_{g_\epsilon}$ is determined on  $C^*_{II}$ as soon as $V_{g_\epsilon}$ is known on
\begin{equation*}
E_1:=\{(x,s)\in E:s>\log(K)\}.
\end{equation*}
This leaves us with two key questions:
\begin{itemize}
\item How should one choose $g_\epsilon$?
\item  Given $g_\epsilon$, what does $V_{g_\epsilon}(x,s)$ look like for $(x,s)\in E_1$?
\end{itemize}
These questions can be answered heuristically in the spirit of the method applied in Section 3 of~\cite{maximality_principle}, but adapted to the case when $X$ is a spectrally negative L\'evy processes (rather than a diffusion). More precisely, as we shall see in more detail in Section~\ref{candidate_solution}, the general theory of optimal stopping~\mbox{\cite{peskir, optimal_stopping_rules}} together with the principle of smooth and continuous fit~\cite{mikhalevich,pes_shir,peskir,some_remarks} suggest that $g_\epsilon$ should be solution to the ordinary differential equation
\begin{equation}
g_\epsilon^\prime(s)=1-\frac{e^s\Z{q}{g_\epsilon(s)}}{(e^s-K)q\W{q}{g_\epsilon(s)}}\quad\text{on }(\log(K),\epsilon).\label{ode_eq}
\end{equation} 
and that $V_{g_\epsilon}(x,s)=(e^{s\wedge\epsilon}-K)\Z{q}{x-s+g_\epsilon(s)}$ for $(x,s)\in E_1$. Note that there might be many solutions to~\eqref{ode_eq} without an initial/boundary condition. However, we are specifically looking for the solution satisfying $\lim_{s\uparrow\epsilon}g_\epsilon(s)=0$. Summing up, we have suggested/found a candidate stopping time $\tau_{g_\epsilon}$ and candidate value \mbox{function $V_{g_\epsilon}$.}\\
\indent As for the case $q>0\vee\psi(1)$ and $\epsilon=\infty$, one might let $\epsilon$ tend to infinity which informally yields a candidate stopping time of the form~\eqref{expect} with $g_\epsilon$ replaced with $g_\infty$, where $g_\infty$ should satisfy~\eqref{ode_eq}, but on $(\log(K),\infty)$ instead of $(\log(K),\epsilon)$. The corresponding value function $V_{g_\infty}$ is then expected to be of the form $V_{g_\infty}(x,s)=(e^s-K)\Z{q}{x-s+g_\infty(s)}$ for $(x,s)\in E_1$. If we are to identify $g_\infty$ as a solution to~\eqref{ode_eq}, we need an initial/boundary condition which in this case can be found as follows. For $s\gg K$ the payoff in~\eqref{problem1} resembles the payoff of the Shepp-Shiryaev optimal stopping problem~\cite{kyprianou,exitproblems, maximum_process} and hence we expect $s\mapsto s-g_\infty(s)$ to look similar to the optimal boundary of the Shepp-Shiryaev optimal stopping problem for $s\gg K$. Therefore, we expect that $\lim_{s\uparrow\infty}g_\infty(s)=k^*$, where $k^*>0$ is the unique root of the equation $\Z{q}{s}-q\W{q}{s}=0$ (cf.~\cite{maximum_process}).\\
\indent These heuristic arguments are made rigorous in the next section.

\section{Main results}\label{main_results}
\subsection{The different solutions of the ODE}\label{sub_ode}
In this subsection we investigate, for $q>0$, the solutions of the ordinary differential equation
\begin{equation}
g^\prime(s)=1-\frac{e^s\Z{q}{g(s)}}{(e^s-K)q\W{q}{g(s)}}\label{aa}
\end{equation}
whose graph lies in
\begin{equation*}
U:=\{(s,H)\in\R^2:s>\log(K),H>0\}.
\end{equation*}
These solutions will, as already hinted in the previous section, play an important role. But before we analyse~\eqref{aa}, recall that the requirement $\W{q}{0+}<q^{-1}$ is the same as asking that either $X$ is of unbounded variation or $X$ is of bounded variation with $\mathtt{d}>q$. Similarly, the condition $\W{q}{0+}\geq q^{-1}$ means that $X$ is of bounded variation with $0<\mathtt{d}\leq q$. Also note that $\W{q}{0+}\geq q^{-1}$ implies $q\geq\mathtt{d}>\psi(1)$.\\
\indent The existence of solutions to~\eqref{aa} and their behaviour under the different parameter regimes is summarised in the next result.

\begin{lem}\label{ode1}
Assume that $q>0$. 
 For $\epsilon\in(\log(K),\infty)$, we have the following. 
{\renewcommand{\theenumi}{\alph{enumi}}
\renewcommand{\labelenumi}{(\theenumi)}
\begin{enumerate}
\item\label{beh_1} If $q>\psi(1)$ and $\W{q}{0+}< q^{-1}$, then there exists a unique solution $g_\epsilon:(\log(K),\epsilon)\to(0,\infty)$ to~\eqref{aa} such that $\lim_{s\uparrow\epsilon}g_\epsilon(s)=0$.
\item\label{beh_2} If $\W{q}{0+}\geq q^{-1}$ (and hence $q>\psi(1)$), then there exists a unique solution $g_\epsilon:(\log(K),\epsilon\wedge\beta)\to(0,\infty)$ to~\eqref{aa} such that \mbox{$\lim_{s\uparrow\epsilon\wedge\beta}g_\epsilon(s)=0$}. Here, the constant $\beta$ is given by $\beta:=\log\big(K(1-\mathtt{d}/q)^{-1}\big)\in(0,\infty]$.
\item\label{beh_3} If $q\leq\psi(1)$, then there exists a unique solution \mbox{$g_\epsilon:(\log(K),\epsilon)\to(0,\infty)$} to~\eqref{aa} such that $\lim_{s\uparrow\epsilon}g_\epsilon(s)=0$.
\end{enumerate}}
For $\epsilon=\infty$, we have in particular: 
{\renewcommand{\theenumi}{\alph{enumi}}
\renewcommand{\labelenumi}{(\theenumi)}
\begin{enumerate}
\setcounter{enumi}{3}
\item\label{beh_4} If $q>\psi(1)$ and $\W{q}{0+}<q^{-1}$, then there exists a unique solution $g_\infty:(\log(K),\infty)\to(0,\infty)$ to~\eqref{aa} such that $\lim_{s\uparrow\infty}g_\infty(s)=k^*$, where $k^*\in(0,\infty)$ is the unique root of $\Z{q}{s}-q\W{q}{s}=0$.
\item\label{beh_5} If $\W{q}{0+}\geq q^{-1}$ (and hence $q>\psi(1)$), then there exists a unique solution $g_\infty:(\log(K),\beta)\to(0,\infty)$ to~\eqref{aa} such that $\lim_{s\uparrow\beta}g_\infty(s)=0$. The constant $\beta$ is as in~\eqref{beh_2}.
\end{enumerate}}
\noindent Moreover, all the solutions mentioned in~\eqref{beh_1}--\eqref{beh_5} tend to $+\infty$ as \mbox{$s\downarrow\log(K)$}. Finally, note that if $\beta\leq\epsilon$ then the solutions in~\eqref{beh_2} and~\eqref{beh_5} coincide.

Then the qualitative behaviour of the solutions of~\eqref{aa} is displayed in Fig.~\ref{all_pic_1}-\ref{all_pic_3}.
\end{lem}
We will henceforth use the following convention: If a solution to~\eqref{aa} is not defined for all $s\in(\log(K),\infty)$, we extend it to $(\log(K),\infty)$ by setting it equal to zero wherever it is not defined (typically $s\geq\epsilon$).

\begin{SCfigure}[][h]
\includegraphics[scale=0.55]{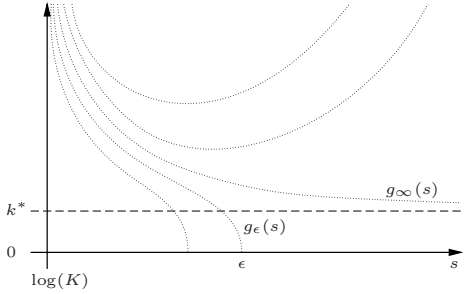}\label{all_pic_1}
\caption{A schematic illustration of the solutions of~\eqref{aa} when $q>\psi(1)$ and $\W{q}{0+}=0$. If \mbox{$q>\psi(1)$} and $\W{q}{0+}\in(0,q^{-1})$, then the solutions look the same except that they hit zero with finite gradient (since \mbox{$\W{q}{0+}>0$}).}
\end{SCfigure}
\begin{SCfigure}[][h]
\includegraphics[scale=0.55]{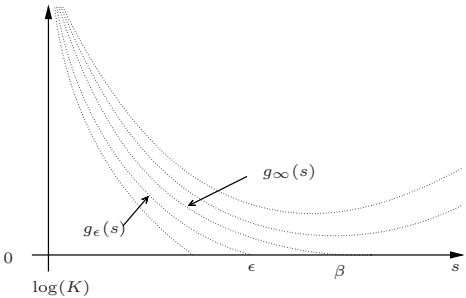}\label{all_pic_2}
\caption{A schematic illustration of the solutions of~\eqref{aa} when $\W{q}{0+}\geq q^{-1}$ and $\epsilon<\beta$.}
\end{SCfigure}
\begin{SCfigure}[][h]
\includegraphics[scale=0.55]{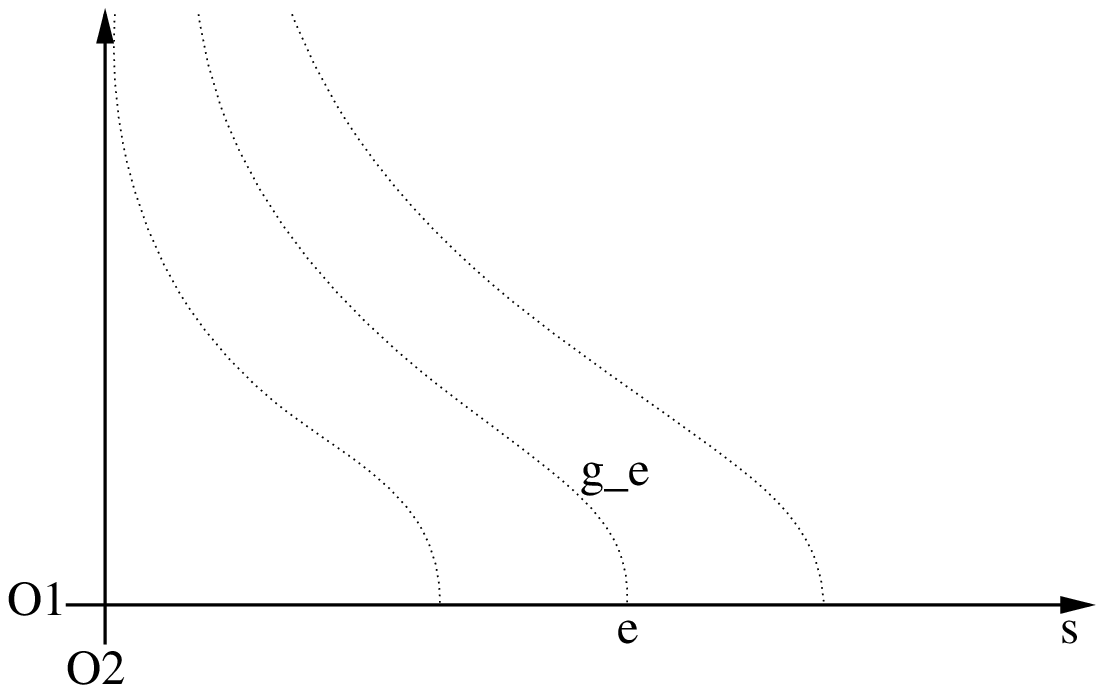}\label{all_pic_3}
\caption{A schematic illustration of the solutions of~\eqref{aa} when $q\leq\psi(1)$ and $\W{q}{0+}=0$. If \mbox{$q\leq\psi(1)$} and $\W{q}{0+}\in(0,q^{-1})$, then the solutions look the same except that they hit zero with finite gradient (since \mbox{$\W{q}{0+}>0$}).}
\end{SCfigure}

\subsection{Verification of the case $q>0$ and $\epsilon\in(\log(K),\infty)$}\label{two}
We are now in a position to state our first main result.
\begin{thm}\label{main_result}
Suppose that $q>0$ and $\epsilon\in(\log(K),\infty)$. Then the solution to~\eqref{problem1} is given by
\begin{equation}
V_\epsilon^*(x,s)=\begin{cases}(e^{s\wedge\epsilon}-K)\Z{q}{x-s+g_\epsilon(s)},&(x,s)\in E_1,\label{value_function}\\
e^{-\Phi(q)(\log(K)-x)}A_\epsilon,&(x,s)\in C^*_{II},\end{cases}
\end{equation}
with value $A_\epsilon\in(0,\infty)$ given by
\begin{equation*}
A_\epsilon:=\E_{\log(K),\log(K)}\big[e^{-q\tau^*_\epsilon}(e^{\overline X_{\tau^*_\epsilon}\wedge \epsilon}-K)\big]=\lim_{s\downarrow\log(K)}(e^s-K)\Z{q}{g_\epsilon(s)},
\end{equation*}
and optimal stopping time
\begin{equation}
\tau_\epsilon^*=\inf\{t\geq 0\,:\,\overline X_t-X_t\geq g_\epsilon(\overline X_t)\text{ and }\overline X_t>\log(K)\},\label{opt_st_time}
\end{equation}
where $g_\epsilon$ is given in Lemma~\ref{ode1}. Moreover,
\begin{equation*}
\P_{x,s}[\tau_\epsilon^*<\infty]=\begin{cases}1,&\text{if }\psi^\prime(0+)\geq 0,\\
e^{-\Phi(q)(\log(K)-x)},&\text{if }\psi^\prime(0+)<0.
\end{cases}
\end{equation*}
\end{thm}
\begin{rem}\label{exc_comp}
With the help of excursion theory, it is possible  to obtain an alternative representation for $V^*_\epsilon(s,s)$ for $\log(K)\leq s<\epsilon$. (See Appendix~\ref{exc_calc} for the relevant computations). Specifically, under the same assumptions as in Theorem~\ref{main_result}, we have
\begin{eqnarray}
V_\epsilon^*(s,s)&=&\int_{s}^{\epsilon\wedge \beta} (e^t-K)\hat f(g_\epsilon(t))\exp\bigg(-\int_{s}^t\frac{\dW{q}{g_\epsilon(u)}}{\W{q}{g_\epsilon(u)}}\,du\bigg)\,dt\label{exc_int}\\
&&+(e^{\epsilon\wedge\beta}-K)\exp\bigg(-\int_{s}^{\epsilon\wedge \beta}\frac{\dW{q}{g_\epsilon(u)}}{\W{q}{g_\epsilon(u)}}\,du\bigg)\notag\label{nosecondterm}
\end{eqnarray}
where $\hat f(u)=\Y{q}{u}{u}{u}{u}$ and we understand $\beta = \infty$ unless $q>0$ and $W^{(q)}(0+)\geq q^{-1}$, in which case we take $\beta = \log\big(K(1-\mathtt{d}/q)^{-1}\big)$ as before. In particular, we can identify the value $A_\epsilon$ as the above expression, setting  $s=\log(K)$.
\end{rem}
\indent Let us now discuss some consequences of Theorem~\ref{main_result}. Firstly, it shows that if $\psi^\prime(0+)\geq 0$ the stopping problem has an optimal solution in the smaller class of $[0,\infty)$-valued $\mathbb{F}$-stopping times. On the other hand, if there is a possibility that the process $X$ drifts to $-\infty$ before reaching $\log(K)$, which occurs exactly when $\psi^\prime(0+)<0$, then the probability that $\tau_\epsilon^*$ is infinite is strictly positive and $\tau_\epsilon^*$ is only optimal in the class of $[0,\infty]$-valued $\mathbb{F}$-stopping times.\\
\indent Secondly, when $\W{q}{0+}\geq q^{-1}$ or, equivalently, $X$ is of bounded variation with $q\geq\mathtt{d}$, the result shows that  $
g_\epsilon(s)$ hits the origin at $\epsilon\wedge\beta$, where $\beta=\log\big(K(1-\mathtt{d}/q)^{-1}\big)$ (see Fig.~\ref{boundary1}). Intuitively speaking, if $\beta<\epsilon$, the discounting is so strong that it is best to stop even before reaching the level $\epsilon$. On the other hand, if $\beta\geq\epsilon$, it would be better to wait longer, but as there is a cap we are forced to stop as soon as we have reached it.\\
\indent As already observed in~\cite{maximum_process}, it is also the case in our setting that, if $\W{q}{0+}<q^{-1}$, the slope of $g_\epsilon$ at $\epsilon\wedge\beta$ (and hence the shape of the optimal boundary $s\mapsto s-g_\epsilon(s)$) changes according to the path variation of $X$. Specifically, it holds that
\begin{equation*}
\lim_{s\uparrow\epsilon\wedge\beta}g_\epsilon^\prime(s)=\begin{cases}
-\infty,&\text{if $X$ is of unbounded variation.}\\
1-\frac{e^{\epsilon\wedge\beta}\mathtt{d}}{(e^{\epsilon\wedge\beta}-K)q},&\text{if $X$ is of bounded variation.}
\end{cases}
\end{equation*}
\indent Next, introduce the sets
\begin{eqnarray}
&&C^*_I=C^*_{I,\epsilon}:=\{(x,s)\in E\,:\,x>\log(K),x>s-g_\epsilon(s)\},\label{sets}\\
&&D^*=D^*_\epsilon:=\{(x,s)\in E\,:\,s>\log(K), x\leq s-g_\epsilon(s)\}\notag.
\end{eqnarray}
\noindent Two examples of $g_\epsilon$ and the corresponding continuation region $C^*_I\cup C^*_{II}$ and stopping region $D^*$ are pictorially displayed in Fig.~\ref{boundary1}. 
\begin{figure}[h]
\includegraphics[scale=0.5]{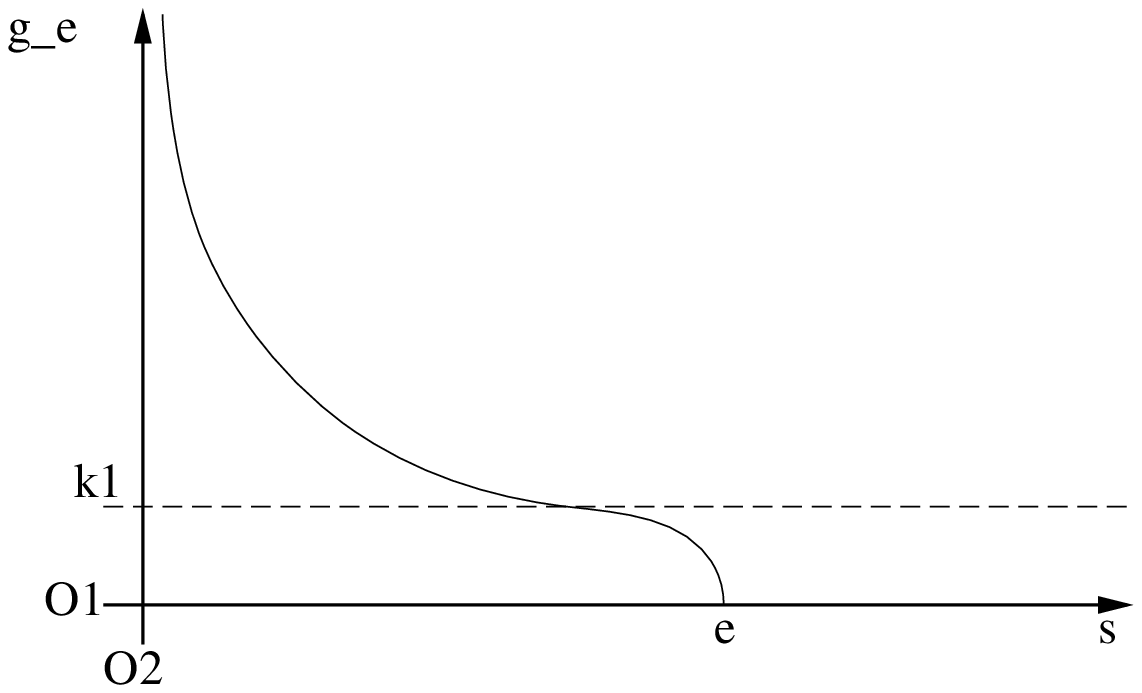}\quad\quad
\includegraphics[scale=0.5]{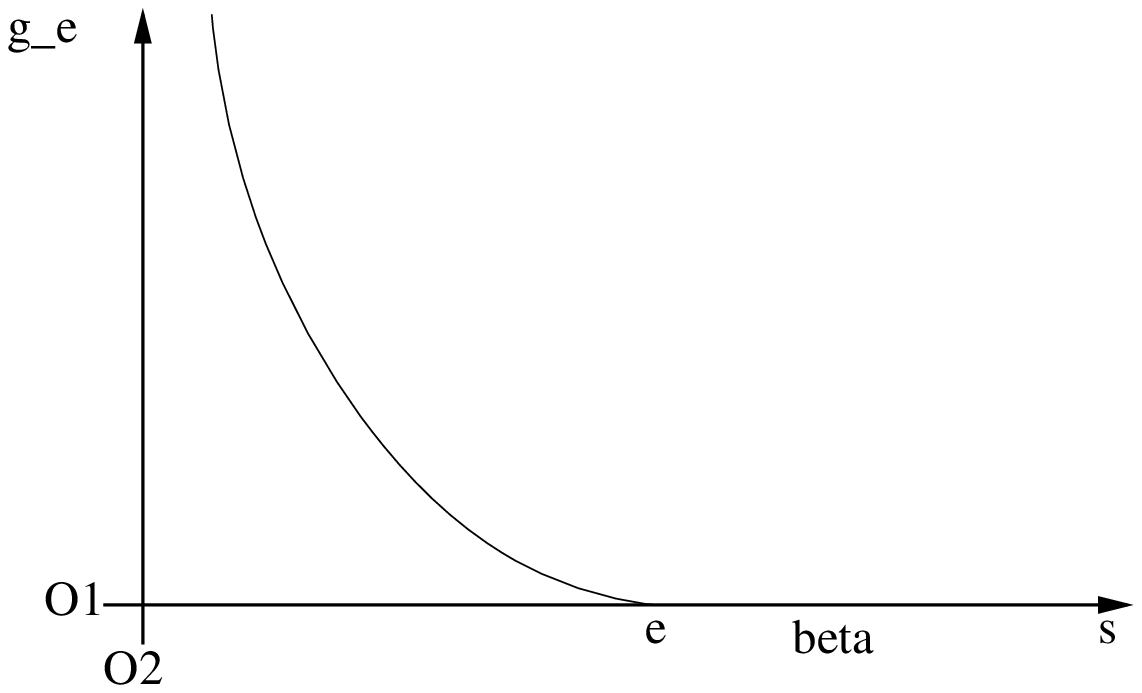}
\includegraphics[scale=0.45]{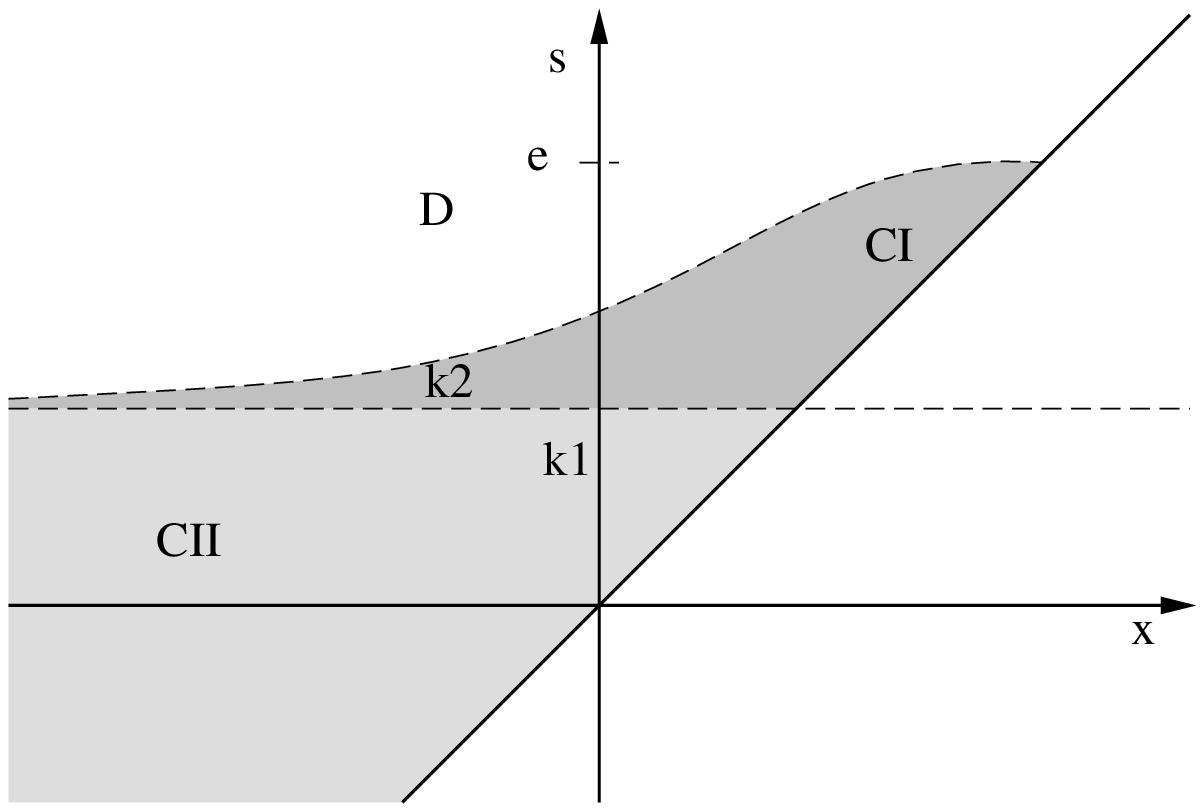}\quad\quad
\includegraphics[scale=0.45]{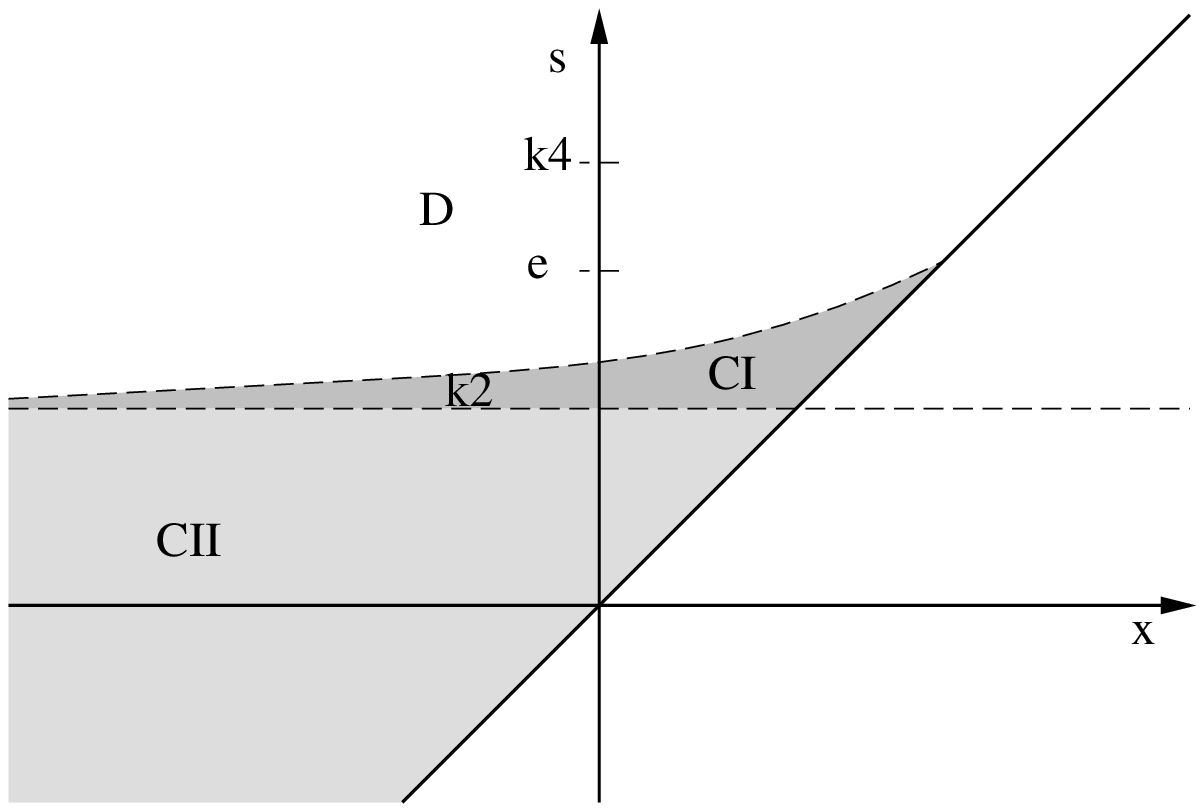}
\caption{For the two pictures on the left it is assumed that $q>0$ and $\W{q}{0+}=0$, whereas on the right it is assumed that $q>0$, $\W{q}{0+}\geq q^{-1}$ and $\epsilon<\beta$.}\label{boundary1}
\end{figure}

\subsection{Verification of the case $q>0\vee\psi(1)$ and $\epsilon=\infty$}\label{one}
The analogous result to Theorem~\ref{main_result} reads as follows.
\begin{thm}\label{main_result_1}
Suppose that $q>0\vee\psi(1)$ and $\epsilon=\infty$. Then the solution to~\eqref{problem1} is given by
\begin{equation}
V_\infty^*(x,s)=\begin{cases}(e^s-K)\Z{q}{x-s+g_\infty(s)},&(x,s)\in E_1,\\
e^{-\Phi(q)(\log(K)-x)}A_\infty,&(x,s)\in C^*_{II},\end{cases}\label{value_function_1}
\end{equation}
with value $A_\infty\in(0,\infty)$ given by
\begin{equation*}
A_\infty:=\E_{\log(K),\log(K)}\big[e^{-q\tau^*_\infty}(e^{\overline X_{\tau^*_\infty}}-K)\big]=\lim_{s\downarrow\log(K)}(e^s-K)\Z{q}{g_\infty(s)},
\end{equation*}
and optimal stopping time
\begin{equation}
\tau_\infty^*=\inf\{t\geq 0\,:\,\overline X_t-X_t\geq g_\infty(\overline X_t)\text{ and }\overline X_t>\log(K)\},\label{opt_st_time_1}
\end{equation}
where $g_\infty$ is given in Lemma~\ref{ode1}. Moreover,
\begin{equation*}
\P_{x,s}[\tau_\infty^*<\infty]=\begin{cases}1,&\text{if }\psi^\prime(0+)\geq 0,\\
e^{-\Phi(q)(\log(K)-x)},&\text{if }\psi^\prime(0+)<0.
\end{cases}
\end{equation*}
\end{thm}
\begin{rem}
As in Remark~\ref{exc_comp}, $V_\infty^*(s,s)$ can be identified as the integral in~\eqref{exc_int} with $\epsilon=\infty$ for $\log(K)\leq s<\epsilon$ in the case $q>0$ and \mbox{$W^{(q)}(0+)\geq q^{-1}$}. Otherwise it is identified as 
\[
V_\infty^*(s,s)=\int_{s}^{\infty} (e^t-K)\hat f(g_\infty(t))\exp\bigg(-\int_{s}^t\frac{\dW{q}{g_\infty(u)}}{\W{q}{g_\infty(u)}}\,du\bigg)\,dt,
\] 
where $\hat f(u)=\Y{q}{u}{u}{u}{u}$ as before. (See again the computations in Appendix B).
In particular, one obtains an alternative expression for $A_\infty$.
\end{rem}
\indent Similarly to Theorem~\ref{main_result} one sees again that if $\psi^\prime(0+)\geq 0$ there is an optimal stopping time in the class of all $[0,\infty)$-valued $\mathbb{F}$-stopping times. Furthermore, let $C^*_I=C^*_{I,\infty}$ and $D^*=D^*_\infty$ denote the same sets as in~\eqref{sets}, but with $g_\infty$ instead of $g_\epsilon$. The (qualitative) behaviour of $g_\infty$ and the resulting shape of the continuation region $C^*_I\cup C^*_{II} $ and stopping region $D^*$ are illustrated in Fig.~\ref{boundary}.

\begin{figure}[h]
\includegraphics[scale=0.5]{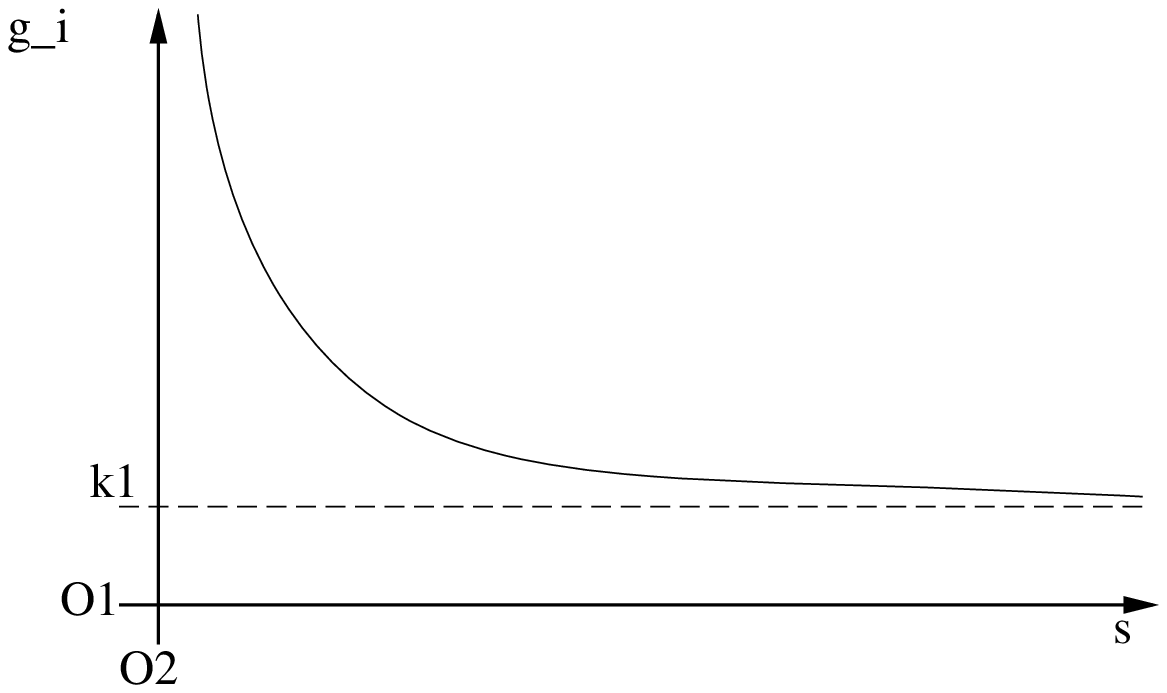}\quad
\includegraphics[scale=0.5]{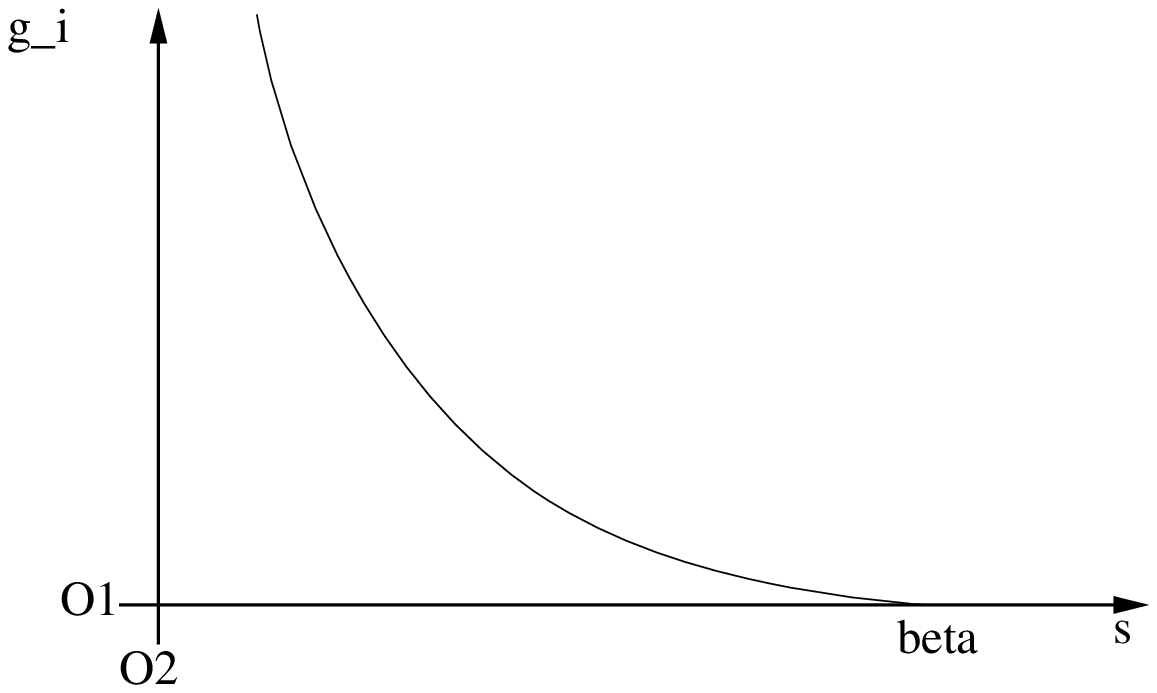}
\includegraphics[scale=0.45]{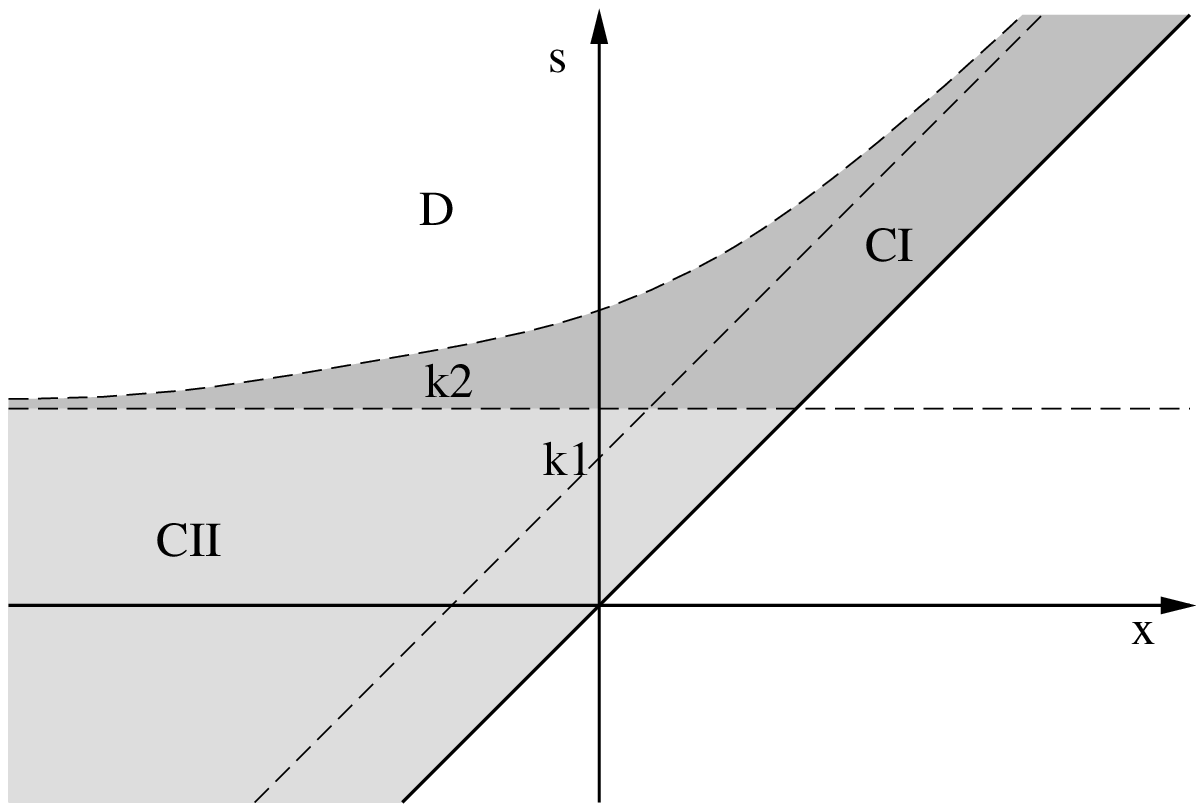}\quad\quad
\includegraphics[scale=0.45]{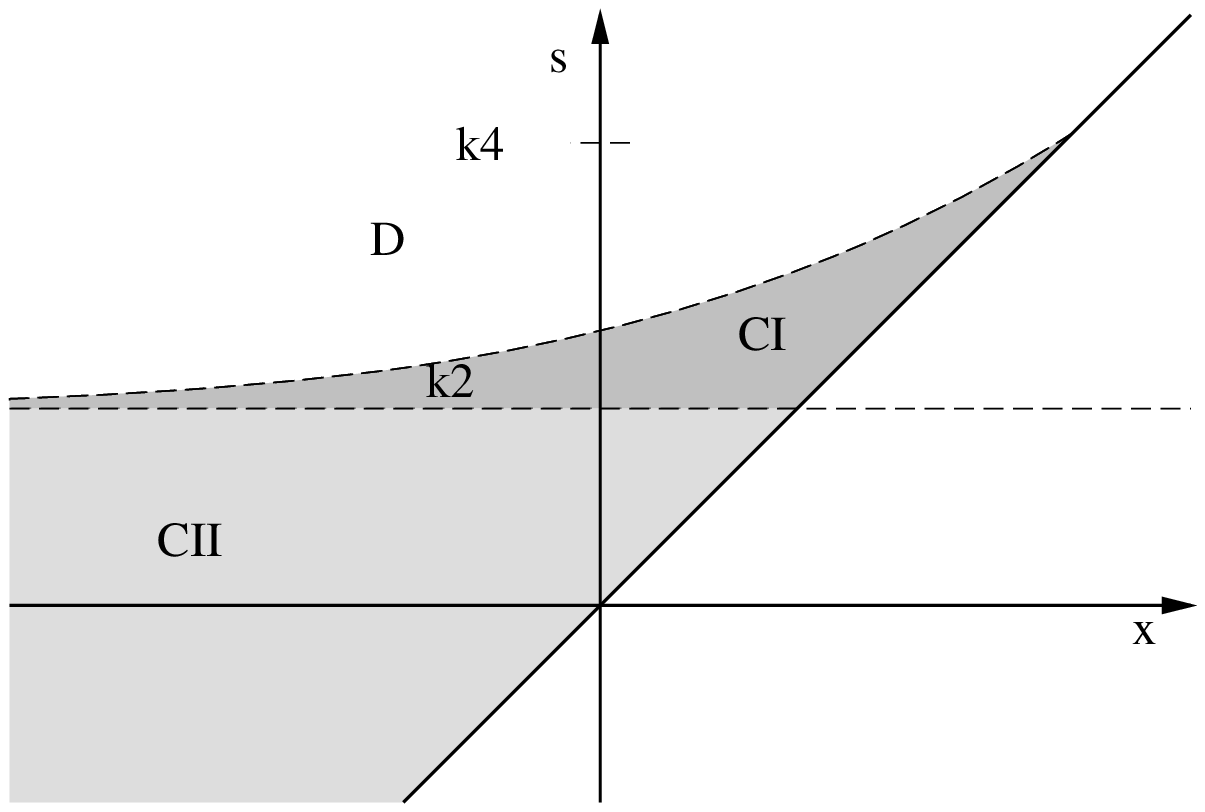}
\caption{For the two pictures on the left it is assumed that $q>0\vee\psi(1)$ and $\W{q}{0+}<q^{-1}$, whereas on the right it is assumed that  $q>0\vee\psi(1)$ and $\W{q}{0+}\geq q^{-1}$.}\label{boundary}
\end{figure}

\subsection{The special cases}
In this subsection we deal with the cases that have not been considered yet, i.e., the special cases (see Section~\ref{regimes}).
\begin{lem}\label{special_cases_1}
Suppose that $q=0$ and $\epsilon\in(\log(K),\infty)$.
{\renewcommand{\theenumi}{\alph{enumi}}
\renewcommand{\labelenumi}{(\theenumi)}
\begin{enumerate}
\item\label{a} When $\psi^\prime(0+)<0$ and $\Phi(0)\neq 1$, then the solution to~\eqref{problem1} is given by
\begin{equation*}
V_\epsilon^*(x,s)=\begin{cases}e^\epsilon-K,&s\geq\epsilon,\\e^s-K+\frac{e^{x\Phi(0)}}{\Phi(0)-1}\big(e^{s(1-\Phi(0))}-e^{\epsilon(1-\Phi(0))}\big)
,&\log(K)\leq s<\epsilon,\\
e^{-\Phi(0)(\log(K)-x)}A_\epsilon,&s<\log(K),\end{cases}
\end{equation*}
where $A_\epsilon:=\frac{K^{\Phi(0)}(K^{1-\Phi(0)}-e^{\epsilon(1-\Phi(0))})}{\Phi(0)-1}$, and $\tau_\epsilon^*=\tau_\epsilon^+$. If $\Phi(0)=1$, then the middle term on the right-hand side in the expression for $V_\epsilon^*(x,s)$ has to be replaced by $e^s-K+e^x(\epsilon-s)$ and $A_\epsilon$ by $K(\epsilon-\log(K))$.\\
\item\label{b} When $\psi^\prime(0+)\geq 0$, then solution to~\eqref{problem1} is given by $V_\epsilon^*\equiv e^\epsilon-K$ and $\tau_\epsilon^*=\tau^+_\epsilon$. 
\end{enumerate}}
\end{lem}
Note that although the optimal stopping time is the same in both parts of Lemma~\ref{special_cases_1}, in~\eqref{a} it attains the value infinity with positive probability, whereas in~\eqref{b} this happens with probability zero. Hence, in~\eqref{b} there is actually an optimal stopping time in the class of finite $\mathbb{F}$-stopping times.
\begin{lem}\label{special_cases_2}
Suppose that $\epsilon=\infty$. 
{\renewcommand{\theenumi}{\alph{enumi}}
\renewcommand{\labelenumi}{(\theenumi)}
\begin{enumerate}
\item\label{first_part} Assume that $q=0$. If $\psi^\prime(0+)<0$ and $\Phi(0)>1$, we have
\begin{equation}
V_\infty^*(x,s)=\begin{cases}e^s-K+\frac{e^{x\Phi(0)+s(1-\Phi(0))}}{\Phi(0)-1},&s\geq\log(K),\\
e^{-\Phi(0)(\log(K)-x)}\frac{K}{\Phi(0)-1},&s<\log(K),\end{cases}
\end{equation}
and the optimal stopping time is given by $\tau_\infty^*=\infty$. On the other hand, if either $\psi^\prime(0+)<0$ and $\Phi(0)\leq 1$ or $\psi^\prime(0+)\geq 0$, then $V_\infty^*(x,s)\equiv\infty$ and $\tau_\infty^*=\infty$. 
\item\label{second_part} When $0<q\leq\psi(1)$, we have $V_\infty^*(x,s)\equiv\infty$.
\end{enumerate}}
\end{lem}
The second part in the Lemma~\ref{special_cases_2} is intuitively clear. If $0<q\leq\psi(1)$, then the average upwards motion of $X$ (and hence $\overline X$) is stronger than the discounting. On the other hand, $\psi^\prime(0+)<0$ means that $X$ will eventually drift to $-\infty$ and thus $X$ will eventually attain its maximum (in the pathwise sense). Of course, we do not know when this happens, but since there is no discounting we do not mind waiting forever. The other cases in Lemma~\ref{special_cases_2} have a similar interpretation.

\subsection{The maximality principle}
The maximality principle was understood as a powerful tool to solve a class of stopping problems for the maximum process associated with a one-dimensional time-homogeneous diffusion~\cite{maximality_principle}. Although we work with a different class of processes, our main results (Lemma~\ref{ode1}, Theorem~\ref{main_result}, Theorem~\ref{main_result_1} and Lemma~\ref{special_cases_2}\eqref{second_part}) can be reformulated through the maximality principle.
\begin{lem}\label{max_pr_1}
Suppose that $q>0$ and $\epsilon\in(\log(K),\infty)$. Define the set
\begin{equation*}
\mathcal{S}:=\big\{g\vert_{(\log(K),\epsilon)}\,\big\vert\, g\text{ is a solution to~\eqref{aa} defined at least on } (\log(K),\epsilon)\big\}.
\end{equation*}
Let $g_\epsilon^*$ be the minimal solution in $\mathcal{S}$. Then the solution to~\eqref{problem1} is given by~\eqref{value_function} and~\eqref{opt_st_time} with $g_\epsilon$ replaced by $g^*_\epsilon$.
\end{lem}
In the case that there is a cap, it cannot happen that the value function becomes infinite. This changes when there is no cap. 
\begin{lem}\label{max_pr_2}
Let $q>0$ and $\epsilon=\infty$.
\begin{enumerate}
\item\label{one_direction} Let $g_\infty^*$ denote the minimal solution to~\eqref{aa} which does not hit zero (whenever such a solution exists). Then the solution to~\eqref{problem1} is given by~\eqref{value_function_1} and~\eqref{opt_st_time_1} with $g_\infty$ replaced by $g_\infty^*$.
\item\label{other_direction} If every solution to~\eqref{aa} hits zero, then the value function in~\eqref{problem1} is given by $V_\infty^*(x,s)\equiv\infty$.
\end{enumerate}
\end{lem}

\begin{rem}
\mbox{ }
\begin{enumerate}
\item We select the minimal solution rather than the maximal one as in~\cite{maximality_principle}, since our functions $g_\epsilon(s)$ are the analogue of $s-g_\epsilon(s)$ in~\cite{maximality_principle}.
\item The ``right'' boundary conditions which were used to select $g_\epsilon$ and $g_\infty$ from the class of solutions of~\eqref{aa} (see Section~\ref{general_observation}) are not used in the formulation of Lemmas~\ref{max_pr_1} and~\ref{max_pr_2}. In fact, by choosing the minimal solution, it follows as a consequence that $g_\epsilon^*$ and $g_\infty^*$ have exactly the ``right'' boundary conditions. Put differently, the ``minimality principle'' is a means of selecting the ``good'' solution from the class of all solutions of~\eqref{aa}. This is merely a reformulation of~\cite{maximality_principle} in our specific setting. 
\item A similar observation  is contained in~\cite{cox}, but in  a slightly different setting.
\item If $\epsilon=\infty$, the solutions to~\eqref{aa} that hit zero correspond to the so-called ``bad-good'' solutions in~\cite{maximality_principle}; ``bad'' since they do not give the optimal boundary, ``good'' as they can be used to approximate the optimal boundary.
\end{enumerate}
\end{rem}

\section{Guess via principle of smooth and continuous fit}\label{candidate_solution}
Our proofs are essentially based on a ``guess and verify'' technique. Here we provide the missing details from Section~\ref{general_observation} on how to ``guess'' a candidate solution. The following presentation is an adaptation of the argument of Section 3 of~\cite{maximality_principle} to our setting.\\
\indent Assume that $q>0$ and $\epsilon\in(\log(K),\epsilon)$. Let $g_\epsilon:(\log(K),\epsilon)\rightarrow(0,\infty)$ be continuously differentiable and define the stopping time $\tau_{g_\epsilon}$ as in~\eqref{expect} and let $V_{g_\epsilon}$ be as in~\eqref{val}. For simplicity assume from now on that $X$ is of unbounded variation (if $X$ is of bounded variation  a similar argument based on the principle of continuous fit applies, see~\cite{pes_shir,some_remarks,peskir} ). From the general theory of optimal stopping, \cite{peskir, optimal_stopping_rules}, we would expect that $V_{g_\epsilon}$ satisfies for $(x,s)\in E$ such that $\log(K)<s<\epsilon$ the system
\begin{eqnarray}
&\Gamma V_{g_\epsilon}(x,s)=qV_{g_\epsilon}(x,s)&\text{for $s-g_\epsilon(s)<x<s$ with $s$ fixed},\notag\\
&\frac{\partial V_{g_\epsilon}}{\partial s}(x,s)\big\vert_{x=s-}=0&\text{(normal reflection),}\label{system_1}\\
&V_{g_\epsilon}(x,s)\vert_{x=(s-g_\epsilon(s))+}=e^s-K&\text{(instantaneous stopping),}\notag
\end{eqnarray} 
where $\Gamma$ is the infinitesimal generator of the process $X$ under $\P$. In addition, the principle of smooth fit (cf.~\cite{mikhalevich,peskir}) suggests that the system above should be complemented by
\begin{equation}
\frac{\partial V_{g_\epsilon}}{\partial x}(x,s)\big\vert_{x=(s-g_\epsilon(s))+}=0\quad\text{(smooth fit),}\label{s_fit}
\end{equation}
Note that the smooth fit condition is not necessarily part of the general theory, it is imposed since by the ``rule of thumb'' outlined in Section 7 in~\cite{some_remarks} one suspects it should hold in this setting because of path regularity. This belief will be vindicated when we show that system~\eqref{system_1} and~\eqref{s_fit} leads to the desired solution.
Applying the strong Markov property at $\tau^+_s$ and using~\eqref{scale1} and~\eqref{scale2} shows that
\begin{eqnarray*}
V_{g_\epsilon}(x,s)
&=&(e^s-K)\bigg(\Z{q}{x-s+g_\epsilon(s)}-\W{q}{x-s+g_\epsilon(s)}\frac{\Z{q}{g_\epsilon(s)}}{\W{q}{g_\epsilon(s)}}\bigg)\\
&&+\frac{\W{q}{x-s+g_\epsilon(s)}}{\W{q}{g_\epsilon(s)}}V_{g_\epsilon}(s,s).
\end{eqnarray*}
Furthermore, the smooth fit condition~\eqref{s_fit} implies
\begin{eqnarray*}
0&=&\lim_{x\downarrow s-g_\epsilon(s)}\frac{\partial V_{g_\epsilon}}{\partial x}(x,s)\\
&=&\lim_{x\downarrow s-g_\epsilon(s)}\frac{\dW{q}{x-s+g_\epsilon(s)}}{\W{q}{g_\epsilon(s)}}\big(V_{g_\epsilon}(s,s)-(e^s-K)\Z{q}{g_\epsilon(s)}\big).
\end{eqnarray*}
By~\eqref{derivativeatorigin} the first factor tends to a strictly positive value or infinity which shows that $V_{g_\epsilon}(s,s)=(e^s-K)\Z{q}{g_\epsilon(s)}$. This would mean that for all \mbox{$(x,s)\in E$} such that $\log(K)<s<\epsilon$ we have
\begin{equation}
V_{g_\epsilon}(x,s)=(e^s-K)\Z{q}{x-s+g_\epsilon(s)}.\label{candidateV}
\end{equation}
Finally, using the normal reflection condition shows that our candidate function $g_\epsilon$ should satisfy the first-order differential equation
\begin{equation}
g_\epsilon^\prime(s)=1-\frac{e^s\Z{q}{g_\epsilon(s)}}{(e^s-K)q\W{q}{g_\epsilon(s)}}\quad\text{on $(\log(K),\epsilon)$}.\label{ode0}
\end{equation}

\section{Example}\label{example}
Suppose that $X_t=(\mu-\frac{1}{2}\sigma^2)t+\sigma W_t$, where $\mu\in\R,\sigma>0$ and $(W_t)_{t\geq 0}$ is a standard Brownian motion. It is well-known that in this case the scale functions are given by
\begin{eqnarray*}
\W{q}{x}=\frac{2}{\sigma^2\delta}e^{\gamma x}\sinh(\delta x)\quad\text{and}\quad\Z{q}{x}=e^{\gamma x}\cosh(\delta x)-\frac{\gamma}{\delta}e^{\gamma x}\sinh(\delta x),
\end{eqnarray*}
on $x\geq 0$, where $\delta(q)=\delta=\sqrt{(\frac{\mu}{\sigma^2}-\frac{1}{2})^2+\frac{2q}{\sigma^2}}$ and $\gamma=\frac{1}{2}-\frac{\mu}{\sigma^2}$. Additionally, \mbox{let $\gamma_1:=\gamma-\delta$} and $\gamma_2:=\gamma+\delta=\Phi(q)$ both of which are the roots of the quadratic equation $\frac{\sigma^2}{2}\theta^2+(\mu-\frac{\sigma^2}{2})\theta-q=0$ and satisfy $\gamma_2>0>\gamma_1$. Using the specific form of $Z^{(q)}$ and $W^{(q)}$ it straightforward to obtain the following result.
\begin{lem}\label{consistency}
Let $\epsilon=\infty$ and assume that $q>\psi(1)$ or, equivalently, $q>\mu$. Then the solution to~\eqref{problem1} is given by
\begin{equation*}
V_\infty^*(x,s)=\begin{cases}
e^s-K,&(x,s)\in D^*,\\
\frac{e^s-K}{\gamma_2-\gamma_1}\big(\gamma_2e^{\gamma_1(x-s+g_\infty(s))}-\gamma_1e^{\gamma_2(x-s+g_\infty(s))}\big),&(x,s)\in C^*_I\\
e^{-\gamma_2(\log(K)-x)}\frac{\gamma_1}{\gamma_1-\gamma_2}A_\infty,&(x,s)\in C^*_{II},
\end{cases}
\end{equation*}
where $A_\infty=\lim_{s\downarrow\log(K)}(e^s-K)e^{\gamma_2g_\infty(s)}$. The corresponding optimal strategy is given by $\tau_\infty^*:=\inf\{t>0:\overline X_t-X_t\geq g_\infty(\overline X_t)\text{ and }\overline X_t>\log(K)\}$, where $g_\infty$ is the unique strictly positive solution to the differential equation
\begin{equation*}
g_\infty^\prime(s)=1-\frac{e^s}{e^s-K}\bigg(\frac{\gamma_2^{-1}e^{\gamma_2g_\infty(s)}-\gamma_1^{-1}e^{\gamma_1g_\infty(s)}}{e^{\gamma_2g_\infty(s)}-e^{\gamma_1g_\infty(s)}}\bigg)\quad\text{on }(\log(K),\infty)
\end{equation*}
such that $\lim_{s\uparrow\infty}g_\infty(s)=k^*$, where the constant $k^*\in(0,\infty)$ is given by
\begin{equation*}
k^*=\frac{1}{\gamma_2-\gamma_1}\log\bigg(\frac{1-\gamma_1^{-1}}{1-\gamma_2^{-1}}\bigg).
\end{equation*}
\end{lem}
\noindent Lemma~\ref{consistency} is nothing other than Theorem 2.5 of~\cite{pedersen} or Theorem 1 of~\cite{guo_shepp} which shows that our results are consistent with the existing literature.

\section{Proof of main results}\label{proofs}
\begin{proof}[Proof of Lemma~\ref{ode1}]
We distinguish three cases:
\begin{itemize}
\item $q>\psi(1)$ and $\W{q}{0+}<q^{-1}$,
\item $\W{q}{0+}\geq q^{-1}$ (and hence $q>\psi(1)$, see beginning of Subsection~\ref{sub_ode}),
\item $\psi(1)\geq q>0$.
\end{itemize}
\subsection*{The case $q>\psi(1)$ and $\W{q}{0+}<q^{-1}$}
The assumptions imply that the function $H\mapsto\Z{q}{H}-q\W{q}{H}$ is strictly decreasing on $(0,\infty)$ and has a unique root $k^*\in(0,\infty)$ (cf. Proposition 2.1 of~\cite{maximum_process}). In particular, $\frac{\Z{q}{H}}{q\W{q}{H}}>1$ for $H<k^*$, $\frac{\Z{q}{H}}{q\W{q}{H}}<1$ for $H>k^*$ and $\frac{\Z{q}{k^*}}{q\W{q}{k^*}}=1$. It is also known that the mapping  $H\mapsto\frac{\Z{q}{H}}{q\W{q}{H}}$ is strictly decreasing on $(0,\infty)$ (cf. first Remark in Section 3 of~\cite{pist}) and that $\lim_{H\to\infty}\frac{\Z{q}{H}}{q\W{q}{H}}=\Phi(q)^{-1}$ (cf. Lemma 1 of~\cite{exitproblems}). We will make use of these properties below.\\
\indent The ordinary differential equation~\eqref{aa} has, at least locally, a unique solution for every starting point  $(s_0,H_0)\in U$ by the Picard-Lindel\"of theorem (cf. Theorem 1.1 in~\cite{hartman}), on account of local Lipschitz continuity of the field. It is well-known that these unique local solutions can be extended to their maximal interval of existence (cf. Theorem 3.1 of~\cite{hartman}). Hence, whenever we speak a solution to~\eqref{aa} from now on, we implicitly mean the unique maximal one. In order to analyse~\eqref{aa}, we sketch its direction field based on various qualitative features of the ODE. The $0$-isocline, that is, the points $(s,H)$ in $U$ satisfying \mbox{$1-\frac{e^s\Z{q}{H}}{(e^s-K)q\W{q}{H}}=0$}, is given by the graph of
\begin{equation}
f(H)=\log\bigg(K\bigg(1-\frac{\Z{q}{H}}{q\W{q}{H}}\bigg)^{-1}\bigg),\quad H\in(k^*,\infty).\label{def_of_f}
\end{equation} 
Using analytical properties of the map $H\mapsto \Z{q}{H}/(q\W{q}{H})$ given at the beginning of the paragraph above, one deduces that $f$ is strictly decreasing on $(k^*,\infty)$ and that $\eta:=\lim_{H\uparrow\infty}f(H)=\log(K(1-\Phi(q)^{-1})^{-1})$ and $\lim_{H\downarrow k^*}f(H)=\infty$. Moreover, the inverse of $f$, which exists due to the strict monotonicity of $f$, will be denoted by $f^{-1}$. Using the $0$-isocline and what was said in the paragraph above, we obtain qualitatively the direction field shown in Fig.~\ref{direction_fields_1}.
\begin{figure}[h]
\begin{center}
\includegraphics[scale=0.48]{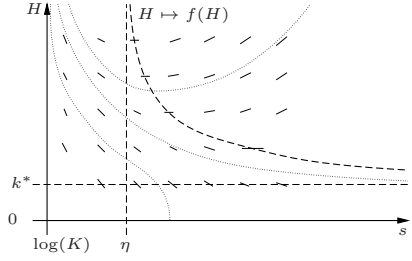}
\caption{A qualitative picture of the direction field when $q>0\vee\psi(1)$ and $\W{q}{0+}=0$. The case when $\W{q}{0+}\in(0,q^{-1})$ is similar except that the solutions (finer line) hit zero with finite slope instead of infinite slope (since $\W{q}{0+}>0$).}\label{direction_fields_1}
\end{center}
\end{figure}

\indent We continue by investigating two types of solutions. Let $s_0>\log(K)$ and let $g(s)$ be the solution such that $g(s_0)=k^*$ which is defined on the maximal interval of existence, say $I_g$, of $g$. From the specific form of the direction field and the fact that solutions tend to the boundary of $U$ (cf. Theorem 3.1 of~\cite{hartman}), we infer that $I_g=(\log(K),\tilde s)$ for some $\tilde s>s_0$, $\lim_{s\uparrow\tilde s}g(s)=0$ and $\lim_{s\downarrow\log(K)}g(s)=\infty$. In other words, the solutions of~\eqref{aa} which intersect the horizontal line $H=k^*$ come from infinity and eventually hit zero (with infinite gradient if $\W{q}{0+}=0$ and with finite gradient if \mbox{$\W{q}{0+}\in(0,q^{-1})$}). Next, suppose that $s_0>\eta$ and let $g(s)$ be the solution such that $g(s_0)=f^{-1}(s_0)$. Similarly to above, we conclude that $I_g=(\log(K),\infty)$, $\lim_{s\uparrow\infty}g(s)=\infty$ and $\lim_{s\downarrow\log(K)}g(s)=\infty$. Put differently, every solution that intersects the $0$-isocline comes from infinity and tends to infinity.\\
\indent Let $\mathcal{S}^-$ be the set of solutions of~\eqref{aa} whose range contains the value $k^*$ and $\mathcal{S}^+$ the set of solutions of~\eqref{aa} whose graph $s\mapsto g(s)$ intersects the $0$-isocline (see Fig.~\ref{direction_fields_1}). Both these sets are non-empty as explained in the previous paragraph. For fixed $s^*>\eta$ define
\begin{eqnarray*}
&&H^*_-:=\sup\{H\in(0,\infty)\,\vert\,\text{there exists $g\in\mathcal{S}^-$ such that $g(s^*)=H$}\}.\\
&&H^*_+:=\inf\{H\in(0,\infty)\,\vert\,\text{there exists $g\in\mathcal{S}^+$ such that $g(s^*)=H$}\}.
\end{eqnarray*} 
It follows that $k^*\leq H^*_-\leq H^*_+\leq f^{-1}(s^*)$ and we claim that $H^*_-=H^*_+$. Suppose this was false and choose $H_1,H_2$ such that $H^*_-<H_1<H_2<H^*_+$. Denote by $g_1$ the solution to~\eqref{aa} such that $g_1(s^*)=H_1$ and by $g_2$ the solutions of~\eqref{aa} such that $g(s^*)=H_2$. Both these solutions must lie between the $0$-isocline and the horizontal line $H=k^*$. In particular, it holds that $I_{g_1}=I_{g_2}=(\log(K),\infty)$ and
\begin{equation}
\lim_{s\to\infty}g_1(s)=\lim_{s\to\infty}g_2(s)=k^*.\label{contradiction}
\end{equation}
Furthermore, set $F(s,H):=1-\frac{e^s\Z{q}{H}}{(e^s-K)q\W{q}{H}}$ for $(s,H)\in U$ and observe that, from earlier remarks,  for fixed $s$, it is an increasing function in $H$. Using this and the fact that $g_1(s)<g_2(s)$ for all $s>\log(K)$ we may write (using the equivalent integral formulation of~\eqref{aa})
\begin{equation*}
g_2(s)-g_1(s)=H_2-H_1+\int_{s^*}^sF(u,g_2(u))-F(u,g_1(u))\,du\geq H_2-H_1>0
\end{equation*}
for $s>\log(K)$. This contradicts~\eqref{contradiction} and hence $H^*_-=H^*_+$. Denote by $g_\infty$ be the solution to~\eqref{aa} such that $g_\infty(s^*)=H^*_-$. By construction, $g_\infty$ lies above all the solutions in $\mathcal{S}^-$ and below all the solutions in $\mathcal{S}^+$. In particular, $I_{g_\infty}=(\log(K),\infty)$ and $\lim_{s\to\infty}g_\infty(s)=k^*$.\\
\indent So far we have found that there are (at least) three types of solutions of~\eqref{aa} and, in fact, there are no more, i.e., any solution to~\eqref{aa} either lies in $\mathcal{S}^-\cup\mathcal{S}^+$ or coincides with $g_\infty$. To see this, note that the graph of $g_\infty$ splits $U$ into two disjoint sets. If $(s,H)\in U$ lies above the graph of $g_\infty$, then the specific form of the field implies that the solution, $g$ say, through $(s,H)$ must intersect the vertical line $s=s^*$ and $g(s^*)>H^*_+$; thus $g\in\mathcal{S}^+$. Similarly, one may deduce that the solution through a point lying below the graph of $g_\infty$ must intersect the horizontal line $H=k^*$ and therefore lies in $\mathcal{S}^-$.\\
\indent Finally, we claim that given $\epsilon>\log(K)$, there exists a unique solution $g_\epsilon$ of~\eqref{aa} such that $I_{g_\epsilon}=(\log(K),\epsilon)$ and $\lim_{s\uparrow\epsilon}g_\epsilon(s)=0$. Indeed, define the sets
\begin{eqnarray*}
s_\epsilon^+&:=&\sup\{s\in(\log(K),\infty)\,\vert\,g\in\mathcal{S}^-\text{ s.t. }I_g\subsetneq(\log(K),\epsilon)\text{ and }g(s)=k^*\},\\
s_\epsilon^-&:=&\inf\{s\in(\log(K),\infty)\,\vert\,g\in\mathcal{S}^-\text{ s.t. }(\log(K),\epsilon)\subsetneq I_g\text{ and }g(s)=k^*\}.
\end{eqnarray*} 
One can then show by a similar argument as above that $s^-_\epsilon=s^+_\epsilon$. The solution through $s^*_+$, denoted $g_\epsilon$, is then the desired one.\\
\indent This whole discussion is summarised pictorially in Fig.~\ref{all_pic_1}.
\subsection*{The case $\W{q}{0+}\geq q^{-1}$}
Similarly to the first case, one sees that under the current assumptions it is still true that $f$ is strictly decreasing on $(0,\infty)$ and $\eta:=\lim_{H\uparrow\infty}f(H)=\log(K(1-\Phi(q)^{-1})^{-1})$. Moreover, recalling that $\W{q}{0+}=\mathtt{d}^{-1}$, one deduces that $\lim_{H\downarrow 0}f(H)=\beta$, where
\begin{equation*}
\beta:=\log(K(1-\mathtt{d}/q)^{-1})\in(0,\infty].
\end{equation*} Analogously to the first case, one may use this information to qualitatively draw the direction field which is shown in Fig.~\ref{direction_fields_2}.
\begin{figure}[h]
\begin{center}
\includegraphics[scale=0.48]{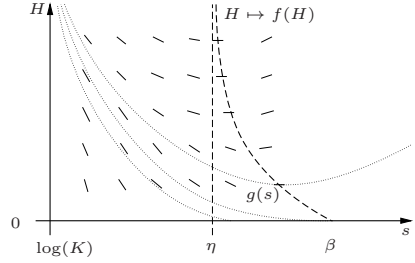}
\caption{A qualitative picture of the direction field when $\W{q}{0+}\geq q^{-1}$. The constants $\eta$ and $\beta$ are given by $\eta=\log(K(1-1/\Phi(q))^{-1})$ and $\beta=\log(K(1-\mathtt{d}/q)^{-1})$. }\label{direction_fields_2}
\end{center}
\end{figure}

\indent As in the first case, one may show that there are again three types of solutions; the ones that intersect the $0$-isocline ($H\mapsto f(H)$) and never hit zero, the ones that hit zero before $\beta$ and the one which lies in between the other two types. One may also show that for a given $\epsilon\in(\log(K),\infty)$ there exists a unique solution $g_\epsilon$ such that $I_{g_\epsilon}=(\log(K),\epsilon\wedge\beta)$ and $\lim_{s\to\epsilon\wedge\beta}g_\epsilon(s)=0$. This is pictorially displayed in Fig.~\ref{all_pic_2}.
 
\subsection*{The case $\psi(1)\geq q>0$}
Under this assumption it holds that $\Phi(q)\leq 1$ which together with equation (8.6) of~\cite{kyprianou} implies that
\begin{equation*}
\Z{q}{H}-q\W{q}{H}\geq\Z{q}{H}-\frac{q}{\Phi(q)}\W{q}{H}>0
\end{equation*}
for $H>0$. This in turn means that $\Z{q}{H}/q\W{q}{H}>1$ for $H>0$. One may again draw the direction field and argue along the same line as above to deduce that all solutions of~\eqref{aa} are strictly decreasing, escape to infinity and hit zero (with infinite gradient if $\W{q}{0+}=0$ and with finite gradient if $\W{q}{0+}\in(0,q^{-1})$).  Again, an argument as in the first case shows that for a given $\epsilon>\log(K)$ there exists a unique solution $g_\epsilon$ such that $I_{g_\epsilon}=(\log(K),\epsilon)$ and $\lim_{s\to\epsilon}g_\epsilon(s)=0$. This was already pictorially displayed in Fig.~\ref{all_pic_3}.
\end{proof}

\begin{proof}[Proof of Theorem~\ref{main_result}]
The proof consists five of steps~\eqref{cond1}-\eqref{cond5} which will imply the result. Before we go through these steps, recall that
\begin{equation}
\limsup_{t\to\infty}e^{-qt}(e^{\overline X_t\wedge\epsilon}-K)=0\quad\P_{x,s}\text{-a.s.}\label{limsup_1}
\end{equation}
for $(x,s)\in E$ and let $\tau^*_\epsilon$ be given as in~\eqref{opt_st_time}. Moreover, define the function
\begin{equation*}
V_\epsilon(x,s):=(e^{s\wedge\epsilon}-K)\Z{q}{x-s+g_\epsilon(s)}
\end{equation*}
for $(x,s)\in E_1=\{(x,s)\in E:s>\log(K)\}$. We claim that
{\renewcommand{\theenumi}{\roman{enumi}}
\renewcommand{\labelenumi}{(\theenumi)}
\begin{enumerate}
\item\label{cond1} $\E_{x,s}[e^{-qt}V_\epsilon(X_t,\overline X_t)]\leq V_\epsilon(x,s)$ for $(x,s)\in E_1$,
\item\label{cond2} $V_\epsilon(x,s)=\E_{x,s}\big[e^{-q\tau^*_\epsilon}(e^{\overline X_{\tau^*_\epsilon}\wedge \epsilon}-K)\big]$ for $(x,s)\in E_1$.
\end{enumerate}}

\subsection*{Verification of~\eqref{cond1}}
We first prove~\eqref{cond1} under the assumption that $X$ is of unbounded variation, that is, $\W{q}{0+}=0$. To this end, let $\Gamma$ be the infinitesimal generator of $X$ and formally define the  function $\Gamma Z^{(q)}:\R\rightarrow\R$ by
\begin{eqnarray*}
\Gamma\Z{q}{x}&:=&-\gamma\dZ{q}{x}+\frac{\sigma^2}{2}\ddZ{q}{x}\\
&&+\int_{(-\infty,0)}\big(\Z{q}{x+y}-\Z{q}{x}-y\dZ{q}{x}1_{\{y\geq -1\}}\big)\,\Pi(dy).
\end{eqnarray*}
For $x<0$ the quantity $\Gamma\Z{q}{x}$ is well-defined and $\Gamma\Z{q}{x}=0$. However, for $x>0$ one needs to check whether the integral part in $\Gamma\Z{q}{x}$ is well-defined. This is done in Lemma A.1 in the Appendix of~\cite{maximum_process} which shows that this is indeed the case. Moreover, as shown in Section 3.2 of~\cite{pist}, it holds that
\begin{equation*}
\Gamma\Z{q}{x}=q\Z{q}{x},\quad x\in(0,\infty).
\end{equation*}
At zero the second derivative of $Z^{(q)}$ does not exist. In this case we understand $\ddZ{q}{0}:=\lim_{x\uparrow 0}\ddZ{q}{x}=0$ and with this definition we have $\Gamma\Z{q}{0}=0$.

\indent Now fix $(x,s)\in E_1$ and define the semimartingale $Y_t:=X_t-\overline X_t+g_\epsilon(\overline X_t)$. Applying an appropriate version of the It\^o-Meyer formula (cf.~Theorem 71,~Ch.~VI of~\citep{protter}) to $\Z{q}{Y_t}$ yields $\P_{x,s}$-a.s.
\begin{eqnarray*}
\Z{q}{Y_t}&=&\Z{q}{x-s+g_\epsilon(s)}+m_t+\int_0^t\Gamma\Z{q}{Y_u}\,du\\
&&+\int_0^t\dZ{q}{Y_u}(g_\epsilon'(\overline X_u)-1)\,d\overline X_u,
\end{eqnarray*}
where
\begin{eqnarray*}
&&\lefteqn{m_t=\int_{0+}^t\sigma\dZ{q}{Y_{u-}}dB_u+\int_{0+}^t\dZ{q}{Y_{u-}}dX^{(2)}_u}\\
&&+\sum_{0<u\leq t}\Delta\Z{q}{Y_u}-\Delta X_u\dZ{q}{Y_{u-}}1_{\{\Delta X_u\geq-1\}}\\
&&-\int_0^t\int_{(-\infty,0)}\Z{q}{Y_{u-}+y}-\Z{q}{Y_{u-}}-y\dZ{q}{Y_{u-}}1_{\{y\geq -1\}}\,\Pi(dy)du
\end{eqnarray*}
and $\Delta X_u=X_u-X_{u-},\,\Delta\Z{q}{Y_u}=\Z{q}{Y_u}-\Z{q}{Y_{u-}}$. By the boundedness of $Z^{(q)\prime}$ on $(-\infty,g(s)]$ the first two stochastic integrals on the right are zero-mean square-integrable martingales and by the compensation formula (cf.~Corollary 4.6 of~\citep{kyprianou}) the third and fourth term constitute a zero-mean  square-integrable martingale. Next, use stochastic integration by parts for semimartingales (cf.~Corollary 2 of Theorem 22, Ch.~II of~\citep{protter}) to deduce that $\P_{x,s}$-a.s.
\begin{eqnarray}
e^{-qt}V_\epsilon(X_t,\overline X_t)&=&V_\epsilon(x,s)+M_t+\int_0^te^{-qu}(e^{\overline X_u\wedge\epsilon}-K)(\Gamma-q)\Z{q}{Y_u}\,du\notag\\
&&+\int_0^te^{-qu}(e^{\overline X_u\wedge\epsilon}-K)\dZ{q}{Y_u}(g'(\overline X_u)-1)\,d\overline X_u\label{final}\\
&&+\int_0^te^{-qu+\overline X_u}\Z{q}{Y_u}1_{\{\overline X_u\leq\epsilon\}}\,d\overline X_u\notag
\end{eqnarray}
where $M_t=\int_{0+}^te^{-qu}(e^{\overline X_u\wedge\epsilon}-K)\,dm_u$ is a zero-mean square-integrable martingale. The first integral is nonpositive since $(\Gamma-q)\Z{q}{y}\leq 0$ for all $y\in\R$. The last two integrals vanish since the process $\overline X_u$ only increments when $\overline X_u=X_u$ and by definition of $g_\epsilon$. Thus, taking expectations on both sides of~\eqref{final} gives~\eqref{cond1} if $X$ is of unbounded variation.\\
\indent If $\W{q}{0+}\in(0,q^{-1})$ or $\W{q}{0+}\geq q^{-1}$ (X has bounded variation), then the It\^o-Meyer formula is nothing more than an appropriate version of the change of variable formula for Stieltjes integrals and one may obtain~\eqref{cond1} in the same way as above. The only change worth mentioning is that the generator of $X$ takes the form
\begin{equation*}
\Gamma\Z{q}{x}=\mathtt{d}\dZ{q}{x}+\int_{(-\infty,0)}\big(\Z{q}{x+y}-\Z{q}{x}\big)\Pi(dy).
\end{equation*}
The last expression is well-defined and $\Gamma Z^{(q)}$ satisfies all the required properties in the proof by the results in the Appendix of~\cite{maximum_process}. This completes the proof of~\eqref{cond1}.

\subsection*{Verification of~\eqref{cond2}}
Recalling that $(\Gamma-q)\Z{q}{y}=0$ for $y>0$, we see from~\eqref{final} that $\E_{x,s}\big[e^{-q(t\wedge\tau^*_\epsilon)}V(X_{t\wedge\tau^*_\epsilon},\overline X_{t\wedge\tau^*_\epsilon})\big]=V_\epsilon(x,s)$ and hence~\eqref{cond2} follows by dominated convergence.\\ 

Next, recall $A_\epsilon:=\E_{\log(K),\log(K)}[e^{-q\tau^*_\epsilon}(e^{\overline X_{\tau^*_\epsilon}\wedge\epsilon}-K)]\label{defA_eps}$ and note that
\begin{equation*}
A_\epsilon=\lim_{s\downarrow\log(K)}\E_{s,s}[e^{-q\tau^*_\epsilon}(e^{\overline X_{\tau^*_\epsilon}\wedge\epsilon}-K)]=\lim_{s\downarrow\log(K)}(e^s-K)\Z{q}{g_\epsilon(s)},
\end{equation*}
where in the second equality we have used~\eqref{cond2} on p.~\pageref{cond2}. Now extend the definition of the function $V_\epsilon$ to
\begin{equation}
V_\epsilon(x,s)=\begin{cases}(e^{s\wedge\epsilon}-K)\Z{q}{x-s+g_\epsilon(s)},&(x,s)\in E_1,\\
e^{-\Phi(q)(\log(K)-x)}A_\epsilon,&(x,s)\in C^*_{II}.\end{cases}
\end{equation}We claim that
{\renewcommand{\theenumi}{\roman{enumi}}
\renewcommand{\labelenumi}{(\theenumi)}
\begin{enumerate}
\setcounter{enumi}{2}
\item\label{cond3} $V_\epsilon(x,s)\geq (e^{s\wedge\epsilon}-K)^+$ for $(x,s)\in E$,
\item\label{cond4} $\E_{x,s}[e^{-qt}V_\epsilon(X_t,\overline X_t)]\leq V_\epsilon(x,s)$ for $(x,s)\in E$,
\item\label{cond5} $V_\epsilon(x,s)=\E_{x,s}\big[e^{-q\tau^*_\epsilon}(e^{\overline X_{\tau^*_\epsilon}\wedge\epsilon}-K)\big]$ for $(x,s)\in E$.
\end{enumerate}}
\noindent Condition~\eqref{cond3} is clear from the definition of $Z^{(q)}$ and $V_\epsilon$.
\subsection*{Verification of condition~\eqref{cond4}}
In view of~\eqref{cond1}, it is enough to show~\eqref{cond4} for $(x,s)\in C^*_{II}$. In order to prove this, set $Y_t=e^{-qt}V_\epsilon(X_t,\overline X_t)$ and observe that
\begin{equation*}
\E_{\log(K),\log(K)}[Y_t]=\lim_{s\downarrow\log(K)}\E_{s,s}[Y_t]\leq\lim_{s\downarrow\log(K)}V_\epsilon(s,s),
\end{equation*}
where in the inequality we have used~\eqref{cond1}. Combining this with the strong Markov property, we obtain on \mbox{$\{\tau^+_{\log(K)}<\infty\}$} for $(x,s)\in C^*_{II}$,
\begin{eqnarray*}
\E_{x,s}\Big[Y_t\Big\vert\mathcal{F}_{\tau^+_{\log(K)}}\Big]&=&Y_t1_{\{t\leq \tau^+_{\log(K)}\}}\\
&&+e^{-q\tau^+_{\log(K)}}\E_{\log(K),\log(K)}[Y_{t-u}]\big\vert_{u=\tau^+_{\log(K)}}1_{\{t>\tau^+_{\log(K)}\}}\\
&&\leq Y_t1_{\{t\leq \tau^+_{\log(K)}\}}+e^{-q\tau^+_{\log(K)}}Y_{\tau^+_{\log(K)}}1_{\{t>\tau^+_{\log(K)}\}}\\
&&=Y_{t\wedge\tau^+_{\log(K)}}.
\end{eqnarray*}
Hence, taking expectations on both sides and using~\eqref{limsup_1} shows that, for \mbox{$(x,s)\in C^*_{II}$}, we have $\E_{x,s}[Y_t]\leq\E_{x,s}\big[Y_{t\wedge\tau^+_{\log(K)}}\big]$. Since $Y_{t\wedge\tau^+_{\log(K)}}$ is a $\P_{x,s}$-martingale for $(x,s)\in C^*_{II}$ (see (\ref{*})) the inequality in~\eqref{cond4} follows.

\subsection*{Verification of condition~\eqref{cond5}}
By the strong Markov property, Theorem 3.12 of~\cite{kyprianou} and the definition of $A_\epsilon$ and $V_\epsilon$ we have
\begin{equation*}
\E_{x,s}\big[e^{-q\tau^*_\epsilon}(e^{\overline X_{\tau^*_\epsilon}\wedge\epsilon}-K)^+\big]=e^{-\Phi(q)(\log(K)-x)}A_\epsilon=V_\epsilon(x,s)
\end{equation*}
for $(x,s)\in C^*_{II}$. This together with~\eqref{cond3} gives assertion~\eqref{cond5}.\\

We are now in a position to prove Theorem~\ref{main_result}. Inequality~\eqref{cond4} and the Markov property of $(X,\overline X)$ imply that the process $e^{-qt}V_\epsilon(X_t,\overline X_t)$ is a $\P_{x,s}$-supermartingale for $(x,s)\in E$. Using~\eqref{limsup_1},~\eqref{cond3}, Fatou's Lemma in the second inequality and the supermartingale property of $e^{-qt}V_\epsilon(X_t,\overline X_t)$ and Doob's stopping theorem in the third inequality shows that for $\tau\in\mathcal{M}$,
\begin{eqnarray*}
\E_{x,s}\big[e^{-q\tau}(e^{\overline X_{\tau}\wedge \epsilon}-K)\big]&=&\E_{x,s}\big[e^{-q\tau}(e^{\overline X_{\tau}\wedge \epsilon}-K)1_{\{\tau<\infty\}}\big]\\
&\leq&\E_{x,s}\big[e^{-q\tau}V_\epsilon(X_\tau,\overline X_\tau)1_{\{\tau<\infty\}}\big]\\
&\leq&\liminf_{t\to\infty}\E_{x,s}\big[e^{-q(t\wedge\tau)}V_\epsilon(X_{t\wedge\tau},\overline X_{t\wedge\tau})\big]\\
&\leq&V_\epsilon(x,s).
\end{eqnarray*}
This together with~\eqref{cond5} shows that $V^*_\epsilon=V_\epsilon$ and that $\tau^*_\epsilon$ is optimal.
\end{proof}

\begin{proof}[ Proof of Theorem~\ref{main_result_1}]
Recall that under the current assumptions Lemma~\ref{integrability} in the Appendix implies that
\begin{eqnarray}
&&\limsup_{t\to\infty}e^{-qt}(e^{\overline X_t}-K)^+=0\quad\P_{x,s}\text{-a.s.}\label{limsup}\\
&&\E_{x,s}\Big[\sup_{0\leq t<\infty}e^{-qt+\overline X_t}\Big]<\infty\label{inte}
\end{eqnarray}
for $(x,s)\in E$, from which it follows that
\begin{equation*}
\sup_{\tau\in\mathcal{M}}\E_{x,s}\big[e^{-q\tau}(e^{\overline X_\tau}-K)^+\big]<\infty
\end{equation*}
for $(x,s)\in E$. Also, for $\epsilon\in(\log(K),\infty)$, let $V_\epsilon^*$,$A_\epsilon$, $\tau^*_\epsilon$ and $g_\epsilon$ be as in Theorem~\ref{main_result} and $g_\infty,\tau^*_\infty$ as stated in Theorem~\ref{main_result_1}. An inspection of the proof of Lemma~\ref{ode1} and Theorem 3.2 of~\cite{hartman} show that $g_\infty(s)=\lim_{\epsilon\uparrow\infty}g_\epsilon(s)$ for $s>\log(K)$ which in turn implies that $\lim_{\epsilon\uparrow\infty}\tau^*_\epsilon=\tau^*_\infty\text{ $\P_{x,s}$-a.s.}$ for all $(x,s)\in E$. Furthermore, recall $A_\infty:=\E_{\log(K),\log(K)}[e^{-q\tau^*_\infty}(e^{\overline X_{\tau^*_\infty}}-K)]$ and define
\begin{equation*}
V_\infty(x,s):=\begin{cases}(e^s-K)\Z{q}{x-s+g_\infty(s)},&(x,s)\in E_1,\\
e^{-\Phi(q)(\log(K)-x)}A_\infty,&(x,s)\in C^*_{II}.\end{cases}
\end{equation*} 
Now, using~\eqref{limsup},~\eqref{inte} and dominated convergence, we see that
\begin{equation*}
\lim_{\epsilon\to\infty}A_\epsilon=\lim_{\epsilon\to\infty}\E_{\log(K),\log(K)}\big[e^{-q\tau^*_\epsilon}\big(e^{\overline X_{\tau^*_\epsilon}\wedge\epsilon}-K\big)\big]=A_\infty
\end{equation*}
and
\begin{eqnarray*}
A_\infty&=&\lim_{s\downarrow\log(K)}\E_{s,s}\big[e^{-q\tau^*_\infty}\big(e^{\overline X_{\tau^*_\infty}}-K\big)\big]\\
&=&\lim_{s\downarrow\log(K)}\lim_{\epsilon\to\infty}\E_{s,s}\big[e^{-q\tau^*_\epsilon}\big(e^{\overline X_{\tau^*_\epsilon}}-K\big)\big]\\
&=&\lim_{s\downarrow\log(K)}(e^s-K)\Z{q}{g_\infty(s)}.
\end{eqnarray*}
It follows in particular that $V_\infty(x,s)=\lim_{\epsilon\uparrow\infty}V^*_\epsilon(x,s)$ for $(x,s)\in E$. Next, we claim that
{\renewcommand{\theenumi}{\roman{enumi}}
\renewcommand{\labelenumi}{(\theenumi)}
\begin{enumerate}
\item\label{cond1_1} $V_\infty(x,s)\geq (e^s-K)^+$ for $(x,s)\in E$,
\item\label{cond2_1} $\E_{x,s}[e^{-qt}V_\infty(X_t,\overline X_t)]\leq V_\infty(x,s)$ for $(x,s)\in E$,
\item\label{cond3_1} $V_\infty(x,s)=\E_{x,s}\big[e^{-q\tau^*_\infty}(e^{\overline X_{\tau^*_\infty}}-K)\big]$ for $(x,s)\in E$.
\end{enumerate}}
\noindent Condition~\eqref{cond1_1} is clear from the definition of $Z^{(q)}$ and $V_\infty$. To prove~\eqref{cond2_1}, use Fatou's Lemma and~\eqref{cond1} of the proof of Theorem~\ref{main_result} to show that
\begin{eqnarray*}
\E_{x,s}[e^{-qt}V_\infty(X_t,\overline X_t)]&\leq&\liminf_{\epsilon\to\infty}\E_{x,s}[e^{-qt}V^*_\epsilon(X_t,\overline X_t)]\\
&\leq&\liminf_{\epsilon\to\infty}V^*_\epsilon(x,s)\\
&=&V_\infty(x,s)
\end{eqnarray*}
for $(x,s)\in E$. As for~\eqref{cond3_1}, using~\eqref{limsup},~\eqref{inte} and dominated convergence we deduce that
\begin{eqnarray*}
V_\infty(x,s)&=&\lim_{\epsilon\to\infty}V^*_\epsilon(x,s)\\
&=&\lim_{\epsilon\to\infty}\E_{x,s}\big[e^{-q\tau^*_\epsilon}(e^{\overline X_{\tau^*_\epsilon}\wedge\epsilon}-K)\big]\\
&=&\E_{x,s}\big[e^{-q\tau^*_\infty}(e^{\overline X_{\tau^*_\infty}}-K)\big].
\end{eqnarray*} 
for $(x,s)\in E$.
The proof of the theorem is now completed by using~\eqref{cond1_1}-\eqref{cond3_1} in the same way as in the proof of Theorem~\ref{main_result} to show that $V_\infty^*=V_\infty$ and that $\tau^*_\infty$ is optimal.
\end{proof}

\begin{rem}
Instead of proving Theorem~\ref{main_result_1} via a limiting procedure, it would be possible to prove it analogously to Theorem~\ref{main_result} by going through the It\^o-Meyer formula. We chose to present the prove above as it emphasises that the capped version of~\eqref{problem1} $(\epsilon\in(\log(K),\infty))$, is a building block for the uncapped version of~\eqref{problem1} $(\epsilon=\infty)$ rather than an isolated problem in itself.
\end{rem}

\begin{proof}[Proof of Lemma~\ref{special_cases_1}]
First assume that $\psi^\prime(0+)<0$ and fix $(x,s)\in E$ such that $\log(K)\leq s\leq\epsilon$. Since the supremum process $\overline X$ is increasing and there is no discounting, it follows that
\begin{eqnarray*}
V_\infty^*(x,s)=\E_{x,s}\Big[e^{\overline X_{\tau^+_\epsilon}}\Big]-K=\E_{x,s}[e^{\overline X_\infty\wedge\epsilon}]-K=e^x\E_{0,s-x}[e^{\overline X_\infty\wedge(\epsilon-x)}]-K.
\end{eqnarray*}
The fact that $\psi^\prime(0+)<0$ implies that $\sup_{0\leq u<\infty} X_u$ is exponentially distributed with parameter $\Phi(0)>0$ under $\mathbb{P}_0$ (see equation 8.2 in~\cite{kyprianou}). Thus, if $\Phi(0)\neq 1$, one calculates
\begin{equation*}
V_\infty^*(x,s)=e^s+\frac{e^{x\Phi(0)}}{\Phi(0)-1}\big(e^{s(1-\Phi(0))}-e^{\epsilon(1-\Phi(0))}\big)-K.
\end{equation*}
Similarly, if $\Phi(0)=1$, we have $V_\epsilon^*(x,s)=e^s-K+e^x(\epsilon-s)$.\\
\indent On the other hand, if $(x,s)\in E$ such that $s<\log(K)$ then an application of the strong Markov property at $\tau^+_{\log(K)}$ and Theorem 3.12 of~\cite{kyprianou} gives
\begin{eqnarray*}
V_\infty^*(x,s)&=&\E_{x,s}\Big[\Big(e^{\overline X_{\tau^+_\epsilon}}-K\Big)^+\Big]\\
&=&e^{-\Phi(0)(\log(K)-x)}\E_{\log(K),\log(K)}\Big[e^{\overline X_{\tau^+_\epsilon}}-K\Big]
\end{eqnarray*}
The last expression on the right-hand side is known from the computations above and hence the first part of the proof follows.\\
\indent As for the second part, it is well-known that $\psi^\prime(0+)\geq 0$ implies that \mbox{$\P_{x,s}[\tau^+_\epsilon<\infty]=1$} for $(x,s)\in E$ and since there is no discounting the claim follows.
\end{proof}

\begin{proof}[Proof of Lemma~\ref{special_cases_2}]
The first part follows by taking limits in Lemma~\ref{special_cases_1}, since by monotone convergence we have
\begin{equation*}
V_\infty^*(x,s)=\E_{x,s}\big[(e^{\overline X_\infty}-K)^+\big]=\lim_{\epsilon\uparrow\infty}\E_{x,s}\big[(e^{\overline X_{\tau^+_\epsilon}\wedge\epsilon}-K)^+\big]=\lim_{\epsilon\uparrow\infty}V^*_\epsilon(x,s).
\end{equation*}
As for the second part, note that $V^*_\infty(x,s)\geq\lim_{\epsilon\uparrow\infty}V_\epsilon^*(x,s)$ and hence it is enough to show that the limit equals infinity. To this end, observe that under the current assumptions we have $\lim_{\epsilon\uparrow\infty}g_\epsilon(s)=\infty$ for $s>\log(K)$ (see Lemma~\ref{ode1}\eqref{beh_3}). This in conjunction with the fact that $\lim_{z\to\infty}\Z{q}{z}=\infty$ shows that, for $(x,s)\in E$ such that $s>\log(K)$, 
\begin{equation*}
\lim_{\epsilon\to\infty}V_\epsilon^*(x,s)=\lim_{\epsilon\to\infty}(e^{s\wedge\epsilon}-K)\Z{q}{x-s+g_\epsilon(s)}=\infty.
\end{equation*}
On the other hand, if $(x,s)\in E$ such that $s\leq\log(K)$, the claim follows provided that $\lim_{\epsilon\to\infty}A_\epsilon=\infty$. Indeed, using the strong Markov property and Theorem 3.12 of~\cite{kyprianou} one may deduce that
\begin{equation*}
A_\epsilon\geq\E_{\log(K),\log(K)}\big[e^{-q\tau_s^+}1_{\{\tau_s^+<\tau^*_\epsilon\}}\big]V^*_\epsilon(s,s).
\end{equation*} 
The second factor on the right-hand side increases to $+\infty$ as $\epsilon\uparrow\infty$ by the first part of the proof and thus the proof is complete.
\end{proof}

\appendix

\section{An auxiliary result}
\begin{lem}\label{integrability}
If $q>\psi(1)$ we have for $(x,s)\in E$ that 
\begin{equation*}
\E_{x,s}\Big[\sup_{0\leq t <\infty}e^{-qt+\overline X_t}\Big]<\infty.
\end{equation*}
In particular, $\limsup_{t\to\infty}e^{-qt+\overline X_t}=0$ $\P_{x,s}$-a.s. for $(x,s)\in E$.
\end{lem}

\begin{proof}[Proof of Lemma~\ref{integrability}]
We want to show that
\begin{equation}
\int_0^\infty \P_{x,s}\bigg[\sup_{0\leq t <\infty}e^{-qt+\overline X_t}>y\bigg]\,dy<\infty.\label{integral}
\end{equation}
First note that it is enough to consider the above integral over the interval $(e^s,\infty)$, since for $y<e^s$ the probability inside the integral is equal to one. Next, for $y>e^s$ define $\gamma=\log(y)-x>0$ and write
\begin{eqnarray*}
 &&\P_{x,s}\bigg[\sup_{0\leq t <\infty}e^{-qt+\overline X_t}>y\bigg]\\
 &&=\P\bigg[\sup_{0\leq t <\infty}\bigg(\bigg(\sup_{0\leq u\leq t}X_u\vee (s-x)\bigg)-\gamma-qt\bigg)>0\bigg]\\
&&\leq\P[X_t-qt>\gamma\text{ for some }t]
\end{eqnarray*}
The last expression is the probability that the spectrally negative L\'evy process $\tilde X_t:=X_t-qt$, with Laplace exponent $\psi_{\tilde X}(\theta)=\psi(\theta)-q\theta$, reaches \mbox{level $\gamma$}. Thus,
\begin{equation*}
\P_{x,s}\bigg[\sup_{0\leq t <\infty}e^{-qt+\overline X_t}>y\bigg]\leq e^{-\Phi_{\tilde X}(0)\gamma}=e^{\Phi_{\tilde X}(0)x}y^{-\Phi_{\tilde X}(0)},
\end{equation*}
where $\Phi_{\tilde X}$ is the right-inverse of $\psi_{\tilde X}$. Hence, the integral~\eqref{integral} converges provided $\Phi_{\tilde X}(0)>1$. The latter is indeed satisfied because $\psi_{\tilde X}$ is convex and $\psi_{\tilde X}(1)=\psi(1)-q<0$ by assumption.\\
\indent As for the second assertion, let $\delta>0$ such that $q-\delta>\psi(1)$. By the first part we may now, for $(x,s)\in E$, infer that $\sup_{0\leq t<\infty}e^{-(q-\delta)t+\overline X_t}<\infty$ $\P_{x,s}$-a.s. and hence
\begin{equation}
\limsup_{t\to\infty}e^{-qt+\overline X_t}=\limsup_{t\to\infty}e^{-\delta t}e^{-(q-\delta)t+\overline X_t}=0.
\end{equation}
This completes the proof.
\end{proof}

\section{An excursion theoretic calculation}\label{exc_calc}
Our aim is to compute the value $\E_{s,s}\big[e^{-q\tau^*_\epsilon}\big(e^{\overline X_{\tau^*_\epsilon}\wedge\epsilon}-K\big)\big]$ for $s\in[\log(K),\epsilon)$ with the help of excursion theory (see Remark~\ref{exc_comp}). We shall spend a moment setting up some necessary notation. In doing so, we closely follow p.221--223 in~\cite{exitproblems} and refer the reader to Chapters 6 and 7 in~\citep{bertoin_book} for background reading.
The process $L_t:=\overline X_t$ serves as local time at $0$ for the Markov process $\overline X-X$ under $\P_{0,0}$. Write $L^{-1}:=\{L^{-1}_t:t\geq 0\}$ for the right-continuous inverse of $L$.
The Poisson point process of excursions indexed by local time shall be denoted by $\{(t,\varepsilon_t):t\geq 0\}$, where
\begin{equation*}
\varepsilon_t=\{\varepsilon_t(s):=X_{L^{-1}_t}-X_{L^{-1}_{t-}+s}:0<s< L^{-1}_t-L^{-1}_{t-}\}
\end{equation*}
whenever $L^{-1}_t-L^{-1}_{t-}>0$. Accordingly, we refer to a generic excursion as $\varepsilon(\cdot)$ (or just $\varepsilon$ for short as appropriate) belonging to the space $\mathcal{E}$ of canonical excursions. The intensity measure of the process $\{(t,\varepsilon_t):t\geq 0\}$ is given by $dt\times dn$, where $n$ is a measure on the space of excursions (the excursion measure). A functional of the canonical excursion that will be of interest is $\overline\varepsilon=\sup_{s<\zeta}\varepsilon(s)$, where $\zeta(\varepsilon)=\zeta$ is the length of an excursion. A useful formula for this functional that we shall make use of is the following (cf.~\citep{kyprianou}, Equation (8.18)):
\begin{equation}
n(\overline\varepsilon>x)=\frac{W'(x)}{W(x)}\label{property_ppp_tail}
\end{equation}
provided that $x$ is not a discontinuity point in the derivative of $W$ (which is only a concern when $X$ is of bounded variation, but we have assumed that in this case $\Pi$ is atomless and hence $W$ is continuously differentiable on $(0,\infty))$. Another functional that we will also use is $\rho_a:=\inf\{s>0:\varepsilon(s)>a\}$, the first passage time above $a$ of the canonical excursion $\varepsilon$.\\
\indent We now proceed with the promised calculation involving excursion theory. First, assume that $\log(K)<\epsilon<\infty$ and $\beta = \infty$. Note that for \mbox{$\log(K)\leq s<\epsilon$},
\begin{eqnarray}
\E_{s,s}\big[e^{-q\tau^*_\epsilon}\big(e^{\overline X_{\tau^*_\epsilon}\wedge \epsilon}-K\big)\big]&=&\E_{s,s}\big[e^{-q\tau^*_\epsilon}\big(e^{\overline X_{\tau^*_\epsilon}\wedge \epsilon}-K\big)1_{\{\tau^*_\epsilon<\tau^+_\epsilon\}}\big]\label{two_terms}\\
&&+\E_{s,s}\big[e^{-q\tau^*_\epsilon}\big(e^{\overline X_{\tau^*_\epsilon}\wedge \epsilon}-K\big)1_{\{\tau^*_\epsilon=\tau^+_\epsilon\}}\big]\notag.
\end{eqnarray}
We compute the two terms on the right-hand side separately. An application of the compensation formula in the second equality and using Fubini's theorem in the third equality gives for $\log(K)\leq s< \epsilon$,
\begin{eqnarray*}
&&\E_{s,s}\big[e^{-q\tau^*_\epsilon}\big(e^{\overline X_{\tau^*_\epsilon}\wedge \epsilon}-K\big)1_{\{\tau^*_\epsilon<\tau^+_\epsilon\}}\big]\notag\\
&&=\E\Bigg[\sum_{0<t<\epsilon-s}e^{-qL^{-1}_{t-}}(e^{t+s}-K)1_{\{\overline\varepsilon_{u}\leq g_\epsilon(u+s)\,\forall\,u<t\}}1_{\{\overline\varepsilon_t>g_\epsilon(t+s)\}}e^{-q\rho_{g_\epsilon(s+t)}(\varepsilon_t)}\Bigg]\notag\\
&&=\E\bigg[\int_0^{\epsilon-s}dt\,e^{-qL^{-1}_{t}}(e^{s+t}-K)1_{\{\overline\varepsilon_u\leq g_\epsilon(u+s)\,\forall\,u<t\}}\int_{\mathcal{E}}1_{\{\overline\varepsilon>g_\epsilon(t+s)\}}e^{-q\rho_{g_\epsilon(s+t)}(\varepsilon)}n(d\varepsilon)\bigg]\\
&&=\int_0^{\epsilon-s} (e^{s+t}-K)e^{-\Phi(q)t}\E\Big[e^{-qL^{-1}_t+\Phi(q)t}1_{\{\overline\varepsilon_u\leq g_\epsilon(u+s)\,\forall\,u<t\}}\Big]\hat f(g_\epsilon(t+s))\,dt,\notag
\end{eqnarray*}
where in the first equality the time index runs over local times and the sum is the usual shorthand for integration with respect to the Poisson counting measure of excursions, and $\hat f(u)=\Y{q}{u}{u}{u}{u}$ is an expression taken from Theorem 1 in~\citep{exitproblems}. Next, note that $L^{-1}_t$ is a stopping time and hence a change of measure according to~\eqref{changeofmeasure} shows that the expectation inside the integral can be written as
\begin{equation*}
\P^{\Phi(q)}\big[\overline\varepsilon_u\leq g_\epsilon(u+s)\text{ for all }u<t\big].
\end{equation*}
Using the properties of the Poisson point process of excursions (indexed by local time) and with the help of~\eqref{property_ppp_tail} and~\eqref{scale5} we may deduce
\begin{eqnarray*}
&&\P^{\Phi(q)}\big[\overline\varepsilon_u\leq g_\epsilon(u+s)\text{ for all }u<t\big]\\
&&=\exp\bigg(-\int_0^tn_{\Phi(q)}(\overline\varepsilon>g_\epsilon(u+s))\,du\bigg)\\
&&=\exp\bigg(\Phi(q)t-\int_0^t\frac{\dW{q}{g_\epsilon(u+s)}}{\W{q}{g_\epsilon(u+s)}}\,du\bigg),
\end{eqnarray*}
where $n_{\Phi(q)}$ denotes the excursion measure associated with $X$ under $\P^{\Phi(q)}$. By a change of variables  we finally get for $\log(K)\leq s< \epsilon$,
\begin{eqnarray*}
&&\E_{s,s}\big[e^{-q\tau^*_\epsilon}\big(e^{\overline X_{\tau^*_\epsilon}\wedge \epsilon}-K\big)1_{\{\tau^*_\epsilon<\tau^+_\epsilon\}}\big]\\
&&=\int_s^{\epsilon} (e^t-K)\hat f(g_\epsilon(t))\exp\bigg(-\int_s^t\frac{\dW{q}{g_\epsilon(u)}}{\W{q}{g_\epsilon(u)}}\,du\bigg)\,dt.
\end{eqnarray*}
As for the second term in~\eqref{two_terms}, similarly to the computation of the first term, we obtain for $\log(K)\leq s<\epsilon$,
\begin{eqnarray*}
&&\E_{s,s}\big[e^{-q\tau^*_\epsilon}\big(e^{\overline X_{\tau^*_\epsilon}\wedge\epsilon}-K\big)1_{\{\tau^*_\epsilon=\tau^+_\epsilon\}}\big]\\
&&=(e^\epsilon-K)\E\Big[e^{-qL^{-1}_{\epsilon-s}}1_{\{\overline\varepsilon_t\leq g_\epsilon(t+s)\,\forall\,t<\epsilon-s\}}\Big]\\
&&=(e^\epsilon-K)e^{-\Phi(q)(\epsilon-s)}\P^{\Phi(q)}\big[\overline\varepsilon_t\leq g_\epsilon(t+s)\,\forall\,t<\epsilon-s\big]\\
&&=(e^\epsilon-K)\exp\bigg(-\int_s^\epsilon\frac{\dW{q}{g_\epsilon(u)}}{\W{q}{g_\epsilon(u)}}\,du\bigg).
\end{eqnarray*}
Adding the two terms up gives the expression in Remark~\ref{exc_comp}.

In the case that $\epsilon = \beta = \infty$  the second term on the right hand side of (\ref{two_terms}) is not needed. In the case that $\beta = \log\big(K(1-\mathtt{d}/q)^{-1}\big)<\epsilon$, the cap $\epsilon$ may effectively be replaced by $\beta$ in (\ref{two_terms}).

\end{document}